\documentclass[12pt,reqno]{amsart}
\usepackage{latexsym}
\usepackage{amssymb}  
\usepackage{mathrsfs}
\usepackage{amsmath}
\usepackage{latexsym}
\usepackage{delarray}
\usepackage{amssymb,amsmath,amsfonts,amsthm,
mathrsfs}

\usepackage{enumerate}

\newcommand{\N}{\mathbb N}

\newcommand{\R}{\mathbb R}

\newcommand{\la}{\left\langle}
\newcommand{\ra}{\right\rangle}

\newcommand{\mc}{\mathcal}

\DeclareMathOperator{\diag}{diag}

\numberwithin{equation}{section}
\theoremstyle{plain}
\newtheorem{thm}{Theorem}[section]
\newtheorem{lem}[thm]{Lemma}
\newtheorem{proposition}[thm]{Proposition}

\newtheorem{remark}[thm]{Remark}

\newcounter{tmp}

\usepackage{hyperref}
\usepackage{doi}

\title[Hyperbolic systems with non-diagonalisable principal part]{
Hyperbolic systems with non-diagonalisable principal part and variable multiplicities, I. Well-posedness}

\author[Claudia Garetto]{Claudia Garetto}
\address{
  Claudia Garetto:
  \endgraf
  Department of Mathematical Sciences
  \endgraf
  Loughborough University
  \endgraf
  Loughborough, Leicestershire, LE11 3TU
  \endgraf
  United Kingdom
  \endgraf
  {\it E-mail address} {\rm c.garetto@lboro.ac.uk}
  }
  \author[Christian J\"ah]{Christian J\"ah}
\address{
  Christian J\"ah:
  \endgraf
  Department of Mathematical Sciences
  \endgraf
  Loughborough University
  \endgraf
  Loughborough, Leicestershire, LE11 3TU
  \endgraf
  United Kingdom
  \endgraf
  {\it E-mail address} {\rm c.jaeh@lboro.ac.uk}
  }
  
\author[Michael Ruzhansky]{Michael Ruzhansky}
\address{
  Michael Ruzhansky:
  \endgraf
  Department of Mathematics
  \endgraf
  Imperial College London
  \endgraf
  180 Queen's Gate, London SW7 2AZ
  \endgraf
  United Kingdom
  \endgraf
  {\it E-mail address} {\rm m.ruzhansky@imperial.ac.uk}
  }
  
\dedicatory{Inspired by our colleague and friend Todor Gramchev (1956-2015)}

\thanks{
The third author was supported in parts by EPSRC
 grant EP/R003025/1 and by the Leverhulme Grant RPG-2017-151. No new data was collected
 or generated during the course of research.
}
\date{}


\begin{document}

\begin{abstract}
	 In this paper we analyse the well-posedness of the Cauchy problem for a rather general class of hyperbolic systems with space-time dependent coefficients and with multiple characteristics of variable multiplicity. First, we establish a well-posedness result in anisotropic Sobolev spaces for systems with upper triangular principal part under interesting natural conditions on the orders of lower order terms below the diagonal. Namely, the terms below the diagonal at a distance $k$ to it must be of order $-k$. This setting also allows for the Jordan block structure in the system. Second, we give conditions for the Schur type triangularisation of general systems with variable coefficients for reducing them to the form with an upper triangular principal part for which the first result can be applied. We give explicit details for the appearing conditions and constructions for $2\times 2$ and $3\times 3$ systems, complemented by several examples.
\end{abstract}

\maketitle
			
\tableofcontents
		
\section{Introduction}

The main aim of this paper is to consider the Cauchy problem for hyperbolic systems \begin{equation} \label{eq:CPGO} \left\{\begin{aligned}
		& D_t u = A(t,x,D_x)u + B(t,x,D_x)u + f(t,x), \quad (t,x) \in [0,T] \times \R^n, \\
		& \left. u\right|_{t=0} = u_0, \quad x \in \R^n,
	\end{aligned} \right. \end{equation}  
with the usual notation $D_t=-{\rm i}\partial_t$ and $D_x=-{\rm i}\partial_x$. We assume that $A(t,x,D_x) = \big[ a_{ij}(t,x,D_x) \big]_{i,j=1}^m$ is an $m\times m$ matrix of pseudo-differential operators of order $1$, i.e. $a_{ij} \in C([0,T],\Psi_{1,0}^1(\R^n))$. In the first part of the paper we will also assume that 
	\begin{equation} \label{eq:LambdapN}
	A(t,x,D_x) = \Lambda(t,x,D_x) + N(t,x,D_x),
\end{equation} 
with real-valued symbols in
$$\Lambda(t,x,D_x) = \diag(\lambda_1(t,x,D_x),\lambda_2(t,x,D_x),\dots,\lambda_m(t,x,D_x)),$$ and \begin{equation*}
N(t,x,D_x) = \begin{bmatrix}
	0 & a_{12}(t,x,D_x) & a_{13}(t,x,D_x) & \cdots & a_{1m}(t,x,D_x) \\
	0 & 0 & a_{23}(t,x,D_x) & \cdots & a_{2m}(t,x,D_x) \\
	\vdots & \vdots & \vdots & \cdots & \vdots \\
	0 & 0 & 0 & \dots & a_{m-1m}(t,x,D_x)\\
	0 & 0 & 0 & \dots & 0
\end{bmatrix}.
\end{equation*} 
Finally, we assume that 
$$B(t,x,D_x) = \big[ b_{ij}(t,x,D_x) \big]_{i,j=1}^m, \quad b_{ij} \in C([0,T],\Psi_{1,0}^0(\R^n)),$$ 
is a matrix of pseudo-differential operators of order $0$. We can take any $n\geq 1$ and we can assume that $m\geq 2$ since in the case $m=1$ there are no multiplicities and thus much more is known. It is also well-known that even if all the coefficients in $A$ and $B$ depend only on time, due to multiplicities, the best one can hope for is the well-posedness of the Cauchy problem \eqref{eq:CPGO} in suitable classes of Gevrey spaces.
Thus, the main questions that we address in this paper are:

\begin{itemize}
\item[(Q1)] {\em Under what structural conditions on the zero order part $B(t,x,D_x)$ is the Cauchy problem \eqref{eq:CPGO} well-posed in $C^\infty$ or, even better, in suitable scales of Sobolev spaces?
}
\item[(Q2)] {\em Under what conditions on the general matrix $A(t,x,D_x) = \big[ a_{ij}(t,x,D_x) \big]_{i,j=1}^m$ of first order pseudo-differential operators can we reduce it (microlocally) to another system with $A$ satisfying the upper triangular condition \eqref{eq:LambdapN}? }
\end{itemize} 
Note that this paper is part of a wider analysis of hyperbolic systems with multiplicities. Here we investigate the well-posedness of these systems. In the second part of this paper we plan to carry out the microlocal analysis of their solutions.

In the case of $2\times 2$ systems the questions above have been analysed with the answer to (Q1) given by the following theorem:

\begingroup
\setcounter{tmp}{\value{thm}}
\setcounter{thm}{0} 
\renewcommand\thethm{\Alph{thm}}

\begin{thm}[{\cite[Theorem 7.2]{GramRuz2013}}] Let $m=2$. Suppose that the pseudo-differential operator ${b}_{21}$ is of order not greater than $-1$. Then the Cauchy problem \eqref{eq:CPGO} is well-posed in $C^\infty$. Moreover, it is well-posed in the anisotropic Sobolev space $\begin{bmatrix}
	H^{s_1}(\R^n) \\ H^{s_2}(\R^n)
	\end{bmatrix}$ provided $s_2-s_1 \geq 1$. In that case the solution satisfies the following estimates: \begin{equation*}
	\|u_1(t,\cdot)\|_{H^s} + \|u_2(t,\cdot)\|_{H^{s+1}} \leq c e^{ct} \left( \|u_1^0\|_{H^s} + \|u_2^0\|_{H^{s+1}} \right), \quad 0 \leq t \leq T,
	\end{equation*} for $u_j^0 \in H_{comp}^{s+j-1}(\R^n)$, $j=1,2$, with $c>0$ depending on $s$, $T$, and the support of the initial data.
\end{thm}

The case of systems of general size but for coefficients depending only on $t$ and for $n=1$ was also considered. More precisely, in \cite{GramOrr2011} the authors considered the Cauchy problem \begin{equation} \label{eq:CPGO_orru} \left\{\begin{aligned}
& D_t u = A(t)D_x u + B(t)D_x u + f(t,x), \quad (t,x) \in [0,T] \times \R, \\
& \left. u\right|_{t=0} = 0, \quad x \in \R,
\end{aligned} \right. \end{equation} with $A(t) = \big[ a_{ij}(t) \big]_{i,j=1}^m \in C([0,T])^{m \times m}$ in the form 
\begin{equation*}
A(t) = \Lambda(t) + N(t),
\end{equation*} 
similar to \eqref{eq:LambdapN}. They showed the following result in the absence of lower order terms and for zero Cauchy data:  
\begin{thm}[{\cite[Proposition 1]{GramOrr2011}}] Let $B(t) \equiv 0$ and let $s \in \R$. Then the Cauchy problem \eqref{eq:CPGO_orru} is $C^\infty$-well-posed. Moreover, there exist $r_1$, \dots, $r_{m-1} \in [0,1]$ such that for every
	$f \in C(\R, (H^s(\R))^m)$ identically $0$ at $t=0$ it admits a unique solution $u \in C(\R,(\mathcal S'(\R))^m)$ satisfying \begin{equation*}
	u_m \in C(\R, H^s(\R)), \quad u_{m-j} \in C(\R, (H^{s-r_1-\cdots-r_{j-1}}(\R)),
	\end{equation*} for $j = 1, \dots, m-1$, and identically $0$ at $t=0$. In particular, if $\lambda_j(t) \neq \lambda_k(t)$, $t \in \R$, $1 \leq j < k \leq m$, no loss of anisotropic regularity appears.
\end{thm}

\endgroup
\setcounter{thm}{\thetmp}

The case of (microlocally) {\em diagonalisable} systems of any order with fully variable coefficients was considered by Rozenblum \cite{Roz} under the condition of transversality of the intersecting characteristics. Also allowing the variable multiplicities, this transversality condition was later removed in \cite{KR1,KR2} with sharp $L^p$-estimates for solutions, with further applications to the spectral asymptotics of the corresponding elliptic systems. 

\smallskip
Before stating our main results and collecting some necessary basic notions we give a brief overview of the state of the art for hyperbolic equations and systems. We have a complete understanding of strictly hyperbolic systems, i.e.,  systems without multiplicities, with $C^\infty$-coefficients. This starts with the groundbreaking work of Lax \cite{Lax:57} and H\"ormander \cite{Hor:85} and heavily relies on the modern theory of Fourier integral operators (FIO). Well-posedness is here obtained in the space of distributions $\mathcal{D}'$. There are also well-posedness results for less regular coefficients with respect to $t$. For instance, well-posedness with loss of derivatives has been obtained by Colombini and Lerner in \cite{ColLer} for second order strictly hyperbolic equations with Log-Lipschitz coefficients with respect to $t$ and smooth in $x$.  It is possible to further drop the regularity in $t$ (for instance H\"older), however, this has to be balanced by stronger regularity in $x$ (Gevrey) and leads to more specific (Gevrey) well-posedness results (see \cite{ColKi:02}, \cite{KY:06} and references therein). Paradifferential techniques have been recently used for this kind of strictly hyperbolic equations by Colombini, Del Santo, Fanelli and M\'etivier in \cite{ColMet:1, ColMet:2}.

The analysis of hyperbolic equations with multiplicities (weakly hyperbolic) has started with the seminal paper by Colombini, de Giorgi and Spagnolo \cite{CDS} in the case of coefficients depending only on time.  Profound difficulties in such analysis have been exhibited by Colombini, Jannelli and Spagnolo in \cite{CS} and \cite{CJS2} showing that even the second order wave equation in $\mathbb{R}$ with smooth time-dependent propagation speed (but with multiplicity) and smooth Cauchy data need not be well-posed in $\mathcal{D}'$. However, they turn out to be well-posed in suitable Gevrey classes or spaces of ultradistributions. In the last decades many results were obtained for weakly hyperbolic equations with $t$-dependent coefficients (\cite{ColKi:02, dAKi:05, G:15, GarRuz:1, GR, GarRuz:3, KS:06}, to quote only very few).
More recently, advances in the theory of weakly hyperbolic systems with $t$-dependent coefficients have been obtained for systems of any size in presence of multiplicities with regular or low regular (H\"older) coefficients \cite{G:15, GarRuz:17-1, GarRuz:17-2}. In addition, in \cite{GJ} precise conditions on the lower order terms (Levi conditions) have been formulated to guarantee Gevrey and ultradistributional well-posedness. Previously very few results were known in the field for systems of a certain size ($2\times 2$, $3\times 3$) \cite{dAKiS:04, dAKiS:08} or of a certain form (for instance without lower order terms or with principal part of a certain form) \cite{Yu:05}.

Weakly hyperbolic equations with $x$-dependent coefficients were considered for the first time in the celebrated paper by Bronshtein \cite{Bronshtein}. As shown already in some earlier works by Ivrii, the corresponding Cauchy problem is well-posed under ``almost analytic regularity'', namely, if the coefficients and initial data are in suitable Gevrey classes.  Bronshtein's result was extended to $(t,x)$-dependent scalar equations by Ohya and Tarama  \cite{OT:84} and to systems by Kajitani and Yuzawa \cite{KY:06}. The regularity assumptions are always quite strong with respect to $x$ (Gevrey) and not below H\"older in $t$. See also \cite{ColNi, Ni:06}. Geometrical and microlocal analytic approaches are known for equations or systems under specific assumptions on the characteristics and/or lower order terms. See \cite{MU:2, Hord, IvPet, KR2, PP}, to quote only a few. 
Time-dependent coefficients of low regularity (distributional) have been considered in \cite{GarRuz:ARMA}.

In this paper we will be interested in the case of coefficients depending on both $t$ and $x$ and we will make use of the usual definitions of symbol classes. 
We say that a function $a=a(x,\xi) \in C^{\infty}(\R^n \times \R^n)$ belongs to $S^m_{1,0}( \R^n \times \R^n)$ if there exist constants $C_{\alpha,\beta} >0$ such that \begin{equation*}
		\forall \alpha,\beta \in \N_0^n \,\, : \,\, |\partial_x^\alpha \partial_\xi^\beta a(x,\xi)| \leq C_{\alpha,\beta} \la \xi \ra^{m-|\beta|} \quad \forall (x,\xi) \in \R^n \times \R^n.
	\end{equation*} 
	The set of pseudo-differential operators associated to the symbols in $S^m_{1,0}( \R^n \times \R^n)$ is denoted by $\Psi^m_{1,0}( \R^n \times \R^n)$. 

If there is no question about the domain under consideration, we will abbreviate the symbol- and operator-classes by $S^m_{1,0}$ and $\Psi^m_{1,0}$, respectively, or simply by $S^m$ and $\Psi^m$.

We also denote by $C([0,T], S_{1,0}^m(\R^n \times \R^n))$ the space of all symbols $a(t,x,\xi)\in S_{1,0}^m(\R^n\times\R^n)$ which are continuous with respect to $t$. The set of operators associated to the symbols in $C([0,T], S_{1,0}^m(\R^n \times \R^n))$ is denoted by $C([0,T], \Psi_{1,0}^m(\R^n \times \R^n))$. 

Again, if there is no question about the domain under consideration, we will abbreviate the symbol- and operator-classes by $C S_{1,0}^m$ and $C\Psi^m_{1,0}$, respectively, or simply by $C S^m$ and $C\Psi^m$.

\smallskip
Let us give our main result concerning the first question (Q1) for the systems with the principal part $A$ satisfying the upper triangular condition \eqref{eq:LambdapN}. 
Here, $f_k$, $u_k$ and $u_k^0$, for $k=1,\dots,m$, stand for the components of the vectors 
$f$, $u$ and $u_0$, respectively.

\begin{thm} \label{main_theo_mi}
Let $n\geq 1$, $m\geq 2$, and let
\begin{equation}
\label{eq_CP_mi}
\left\{\begin{array}{l}
D_t u = A(t,x,D_x)u + B(t,x,D_x)u + f(t,x), \quad (t,x) \in [0,T] \times \R^n, \\
\left.u \right|_{t=0} = u_0(x), \quad x \in \R^n,
\end{array} \right.
\end{equation} where $A(t,x,D_x) = [a_{ij}(t,x,D_x)]_{i,j=1}^m$ is an upper-triangular matrix of pseudo-diffe\-ren\-tial operators of order $1$ in the form \eqref{eq:LambdapN}, and $B(t,x,D_x) = [b_{ij}(t,x,D_x)]_{i,j=1}^m$  is a matrix of pseudo-differential operators of order $0$,  continuous with respect to $t$. 
Hence, if 
\begin{center}
the lower order terms $b_{ij}$ belong to $C([0,T], \Psi^{j-i})$ for $i> j$, 
\end{center}
$u^0_k\in H^{s+k-1}(\R^n)$ and $f_k\in C([0,T],H^{s+k-1})$ for $k=1,\dots,m$, then \eqref{eq_CP_mi} has a unique anisotropic Sobolev solution $u$, i.e.,  $u_k\in C([0,T], H^{s+k-1})$ for $k=1,\dots, m$.
\end{thm}

Thus, the main condition of Theorem \ref{main_theo_mi} for the Sobolev well-posedness is that the pseudo-differential operator $b_{ij}$ below the diagonal (i.e. for $i>j$) must be of order $j-i$. In other words, the terms below the diagonal at a distance $k$ to it must be of order $-k$.

In solving the Cauchy problem  \eqref{eq_CP_mi} we will make use of Fourier integral operators depending on the parameter $t\in[0,T]$. Namely, we will work with operators of the type
\begin{equation*}
\int\limits_{0}^{t} \int\limits_{\R^n} e^{{\rm i}\varphi(t,s,x,\xi)} a(t,s,x,\xi) \widehat{g}(s,\xi) d\xi ds
\end{equation*} 
where $\varphi$ is the solution of a certain eikonal equation and the symbol $a$ is determined via asymptotic expansion and transport equations. In Section \ref{sec:auxFIO} we will recall some well-known Sobolev estimates for this type of operators.   

In Section \ref{SEC:wp} we will prove Theorem \ref{main_theo_mi} after we explain its idea in the cases of $m=2$ and $m=3$. 

Consequently, in Section \ref{SEC:Schur} we give an answer to the second question (Q2) above in the form of a suitable variable coefficients extension of the Schur triangularisation.
For constant matrices such a procedure is well known (see e.g. \cite[Theorem 5.4.1]{Bernstein}):  

\begingroup
\setcounter{tmp}{\value{thm}}
\setcounter{thm}{2} 
\renewcommand\thethm{\Alph{thm}}

\begin{thm}[Schur's triagularisation theorem] \label{THM:Schuri}
	Given a (constant) $m \times m$ matrix $A$ with eigenvalues $\lambda_1, \dots, \lambda_m$ in any prescribed order, there is a unitary $m \times m$ matrix $T$ such that $R = T^{-1} A T$ is upper triangular with the diagonal elements $r_{ii} = \lambda_i$. Furthermore, if the entries of $A$ and its eigenvalues are all real, $T$ may be chosen to be real orthogonal.
\end{thm}

\endgroup
\setcounter{thm}{\thetmp}

It follows that $R$ can be written as $D+N$, where $D=\diag(\lambda_1,\dots,\lambda_m)$ and $N$ is a nilpotent upper triangular matrix.

If the matrix $A$ depends on one or several parameters, namely $A=A(t,x,\xi)$, the situation becomes less clear and it is difficult to give a complete description, in particular together with a prescribed regularity of the involved transformation matrices. The regularity of the matrix $A$ and the desire to maintain it through the transformation puts already constrains on the matrix as, in general, the eigenvalues can only be expected to be  Lipschitz continuous in the parameters even if all the entries depend smoothly on the parameters (see, e.g.,  \cite{Bronshtein,Parusinski} and the references therein). In the sequel, we will present some sufficient conditions to ensure the existence of an upper  triangularisation for $A(t,x,\xi)$ which respects its regularity. For example, it will apply to the case when $A$ is a matrix of first order symbols continuous with respect to $t$, i.e., $A(t,x,\xi) \in \big( C S^1 \big)^{m \times m}$.

Our main result for this part of the problem is the following theorem.

\begin{thm}\label{thm:Schuri} 
	Let $A(t,x,\xi) = [a_{ij}]_{i,j=1}^m$, $a_{ij} \in C S^1$, be a matrix with eigenvalues $\lambda_1, \dots,\lambda_{m} \in C S^1$, and let $h_1 , \dots, h_{m-1}  \in \big( C S^0 \big)^m$ be the corresponding eigenvectors. Suppose that for $e_1=(1,0,\cdots,0) \in \R^{m-i+1}$ the condition \begin{equation} \label{eq:CondThmSchuri}
		\la h^{(i)}(t,x,\xi) | e_1 \ra \neq 0, \quad \forall (t,x,\xi) \in [0,T] \times \R^n \times \R^n 
	\end{equation} 
	holds for all $i=1,\dots, m-1$, with the notation for $h^{(i)}$ explained in \eqref{EQ:his}. Then, there exists a matrix-valued symbol $T(t,x,\xi)= [t_{ij}]_{i,j=1}^m$, $t_{ij} \in C S^0$, invertible for $(t,x,\xi) \in [0,T] \times \R^n \times \{ |\xi| \geq M\}$, such that \begin{equation*}
	T^{-1}(t,x,\xi)A(t,x,\xi) T(t,x,\xi) = \Lambda(t,x,\xi) + N(t,x,\xi)
	\end{equation*} for all $(t,x,\xi) \in [0,T] \times \R^n \times \{ |\xi| \geq M \}$, where   \begin{equation*}
	\Lambda(t,x,\xi) = \diag(\lambda_1(t,x,\xi),\lambda_2(t,x,\xi),\dots,\lambda_m(t,x,\xi))
	\end{equation*} and \begin{equation*}
	N(t,x,\xi) = \begin{bmatrix}
	0 & N_{12}(t,x,\xi) & N_{13}(t,x,\xi) & \cdots & N_{1m}(t,x,\xi) \\
	0 & 0 & N_{23}(t,x,\xi) & \cdots & N_{2m}(t,x,\xi) \\
	\vdots & \vdots & \vdots & \cdots & \vdots \\
	0 & 0 & 0 & \dots & N_{m-1m}(t,x,\xi)\\
	0 & 0 & 0 & \dots & 0
	\end{bmatrix},
	\end{equation*}
	and $N$ is a nilpotent matrix with entries in $C S^1$. 
	\end{thm}
Furthermore, there is an expression for the matrix symbol $T$ which will be given in Theorem \ref{thm:Schur}. Also, the assumption \eqref{eq:CondThmSchuri} can be relaxed, see
Remark \ref{rem_gen_cond}.
In Section \ref{SEC:Schur} we will prove this result as well as describe the procedure how to obtain the desired upper triangular form. Moreover, we work out in detail the cases of $m=2$ and $m=3$ clarifying this Schur triangualisation procedure and give a number of examples.

\medskip
The results and techniques of this paper are a natural outgrowth of the paper \cite{GramRuz2013} where the case $m=2$ was considered and to which the results of the present paper reduce in the case of $2\times 2$ systems. It is with great sorrow that we remember the untimely departure of our colleague and friend {\bf Todor Gramchev} who was the inspiration for both  \cite{GramRuz2013} and the present paper.

\section{Well-posedness in anisotropic Sobolev spaces}
\label{SEC:wp}

This section is devoted to proving the well-posedness of the Cauchy problem \eqref{eq:CPGO}. For the reader's convenience we first give a detailed proof in the cases $m=2$ and $m=3$. This will inspire us in proving Theorem \ref{main_theo_mi}. We note that the case $m=2$ has been studied in \cite{GramRuz2013} and we will briefly review its derivartion. However, first we collect a few results about Fourier integral operators that we will need in the sequel.

\subsection{Auxiliary remarks} \label{sec:auxFIO}

In solving the Cauchy problem \eqref{eq:CPGO}, we will deal with solutions of certain scalar pseudo-differential equations. For each characteristic $\lambda_j$ of $A$, we will be denoting by $G^0_j\theta$
 the solution to
  \begin{equation*}
\left\{\begin{array}{l}
D_t w = \lambda_j(t,x,D_x)w + b_{jj}(t,x,D_x)w, \\
w(0,x) = \theta(x),
\end{array}\right.
\end{equation*} 
and by $G_j g$ the solution to  
\begin{equation*}
\left\{\begin{array}{l}
D_t w = \lambda_j(t,x,D_x)w + b_{jj}(t,x,D_x)w + g(t,x), \\
w(0,x) = 0.
\end{array}\right.
\end{equation*}

The operators $G^0_j$ and $G_j$ can be microlocally represented  by Fourier integral operators \begin{equation}\label{EQ:Gj0}
G^0_j \theta(t,x) = \int\limits_{\R^n} e^{{\rm i}\varphi_j(t,x,\xi)} a_j(t,x,\xi) \widehat{\theta}(\xi) d\xi
\end{equation} 
and 
\begin{equation*}
G_j g(t,x) = \int\limits_{0}^{t} \int\limits_{\R^n} e^{{\rm i}\varphi_j(t,s,x,\xi)} A_j(t,s,x,\xi) \widehat{g}(s,\xi) d\xi ds,
\end{equation*} 
with $\varphi_j(t,s,x,\xi)$ solving the eikonal equation
 \begin{equation*}
\left\{
\begin{array}{l}
\partial_t \varphi_j = \lambda_j(t,x,\nabla_x\varphi_j), \\
\varphi_j(s,s,x,\xi) = x \cdot \xi,
\end{array}
\right.
\end{equation*} 
and with the notation
\begin{equation*}
\varphi_j(t,x,\xi) = \varphi_j(t,0,x,\xi).
\end{equation*} 
Here we also have the amplitudes $A_{j,-k}(t,x,\xi)$ of order $-k$, $k$ $\in \N$, giving $A_j \sim \sum_{k=0}^{\infty} A_{j,-k}$, and they satisfy the transport equations with initial data at $t=s$, and we have $a_j(t,x,\xi) = A_j(t,0,x,\xi)$. 

If $a_j\in S^m$, i.e. if the amplitude $a_j$ in \eqref{EQ:Gj0} is a symbol of order $m$, we will write
$G^0_j\in I_{1,0}^{m}.$ However, in the above construction of propagators for hyperbolic equations, we have $a_j\in S^0$, so that $G^0_j\in I_{1,0}^{0}.$

Therefore, we can record the following estimate:

\begin{lem}\label{thm:FIOest}
For any $\sigma\in\mathbb{R}$, for sufficiently small $t$, we have
\begin{equation*}
		\left\| G_j^0\theta(t) \right\|_{H^\sigma} \leq C\|\theta\|_{H^\sigma},\quad
	\left\| G_j g(t) \right\|_{H^\sigma} \leq Ct \|g\|_{L^\infty_s H^\sigma_x}.
	\end{equation*} 
\end{lem}
This statement follows from the continuity of $\lambda_j, \varphi_j, a_j, A_j$ with respect to $t$ and from the $H^\sigma$-boundedness of non-degenerate Fourier integral operators, see e.g. \cite{Duis} (there are also surveys on such questions \cite{Ruz, Ruz-CWI}).

\subsection{The case $m=2$}

To motivate the higher order cases, here we review the construction for $2\times 2$ systems adapting it for the subsequent higher order arguments. Hence, in this subsection we follow the proof in  \cite{GramRuz2013}.
Thus, we consider the system \begin{eqnarray} \label{eq:CP2x2}
\left\{ \begin{array}{l}
D_t u = A(t,x,D_x)u + B(t,x,D_x)u + f(t,x), \quad (t,x) \in [0,T] \times \R^n,\\
\left.u\right|_{t=0} = u_0, \quad x \in \R^n,
\end{array}
\right.
\end{eqnarray} where $u_0(x) = \big[u_1^0(x),u_2^0(x) \big]^T$, $f(t,x) = \big[ f_1(t,x) ,f_2(t,x) \big]^T$, and with the operators $A(t,x,D_x)$ and $B(t,x,D_x)$  given by \begin{equation} \label{eq:A2x2}
A(t,x,D_x) = \begin{bmatrix}
\lambda_1(t,x,D_x) & a_{12}(t,x,D_x) \\
0 & \lambda_2(t,x,D_x) \\
\end{bmatrix} \end{equation} 
and 
\begin{equation*}
B(t,x,D_x) = \begin{bmatrix}
b_{11}(t,x,D_x) & b_{12}(t,x,D_x) \\
b_{21}(t,x,D_x) & b_{22}(t,x,D_x) \\
\end{bmatrix}.
\end{equation*} We suppose that all entries of $A(t,x,D_x)$ belong to $ C \Psi_{1,0}^1$ and all entries of $B(t,x,D_x)$ belong to $ C \Psi_{1,0}^0$. By using the operators $G^0_j$ and $G_j$ introduced in Section \ref{sec:auxFIO}, we can reformulate 
the equations \eqref{eq:CP2x2} as 
\begin{eqnarray}
	\label{eq:u21} u_1 &=& U^0_1 + G_1((a_{12}+b_{12})u_2), \\
	\label{eq:u22} u_2 &=& U^0_2 + G_2(b_{21}u_1), 
\end{eqnarray} where \begin{equation} \label{eq:aux52}
	U_j^0 = G_j^0 u_j^0 + G_j(f_j), \quad j=1,2.
\end{equation} 
Plugging \eqref{eq:u22} in \eqref{eq:u21}, we obtain 
\begin{equation} \label{eq:u1Fin2x2}
		u_1 = \tilde{U}^0_1 + G_1(a_{12}G_2(b_{21}u_1)) + G_1(b_{12}G_2(b_{21}u_1)),
\end{equation} 
where 
\begin{equation} \label{eq:aux51}
	\tilde{U}^0_1 = G_1^0 u_1^0 + G_1(f_1) + G_1((a_{12}+b_{12})U^0_2).
\end{equation}

Using the rules of composition of Fourier integral operators, see e.g. \cite{Duis}, and by Lemma \ref{thm:FIOest}, we get that the operator $G_1 \circ a_{12} \circ G_2 \circ b_{21}$ in \eqref{eq:u1Fin2x2} acts continuously on $H^s$ if it is of order $0$. Since  $a_{12} \in  C \Psi_{1,0}^1$ we therefore need to assume that $b_{21} \in  C \Psi_{1,0}^{-1}$.  

The operator $G_1 \circ b_{12} \circ G_2 \circ b_{21}$ belongs to $ C I_{1,0}^{-1}$ since $b_{21} \in  C \Psi_{1,0}^{-1}$ and $b_{12} \in  C \Psi_{1,0}^{0}$.

We now introduce the following scale of Banach spaces $X(t) :=  C([0,T],H^s)$, $t \in [0,T]$, equipped with the norm \begin{equation*}
	\|u\|_{X(t)} = \sup_{\tau \in [0,t]} \|u(\tau,\cdot)\|_{H^s}.
\end{equation*} Let \begin{equation*}
	\mc G_1^0 u_1:= G_1(a_{12}G_2(b_{21}u_1)) + G_1(b_{12}G_2(b_{21}u_1)).
\end{equation*} 
It follows that \eqref{eq:u1Fin2x2} can be written as
\[
u_1=\tilde{U}^0_1 +\mc G_1^0 u_1.
\]
By composition of  Fourier integral operators and Lemma \ref{thm:FIOest} we have that the 0-order Fourier integral operator $\mc G_1^0$ maps $ C([0,T],H^s)$ continuously into itself and for small time interval it is a contraction, in the sense that there exists $T^\ast\in[0,T]$ such that   
 \begin{equation*}
	\|\mathcal G_1^0(u-v)\|_{X(t)} \leq C_{a,s} T^\ast \|u-v\|_{X(t)},
\end{equation*}
with  $C_{a,s}T^\ast < 1$. Banach's fixed point theorem ensures the existence of a unique fixed point $u_1$ for the map $\mc G_1^0$. Hence, by  assuming that the initial data $\tilde{U}^0_1$ belongs to $ C([0,T^\ast],H^s)$  we conclude that there exists a unique $u_1\in  C([0,T^\ast],H^s) $ solving \eqref{eq:u1Fin2x2}. Note that the same argument proves that the operator $I-\mathcal{G}_1^0$ is invertible on a sufficiently small interval in $t$ since $\mathcal{G}_1^0=I$ at $t=0$. From formula \eqref{eq:aux51} it is clear that in order to get $\tilde{U}^0_1$ to belong to $ C([0,T^\ast],H^s)$ we need to assume that $U^0_2\in H^{s+1}$. Finally, we get $u_2$  by substitution of $u_1$ in \eqref{eq:u22}.

\begin{remark}
\label{rem_2_1}
Note that the constant $T^\ast$ depends only on $a$ and $s$. Thus, the argument above can be iterated by taking $u(T^\ast,x)$ as new initial data. In this way one can cover an arbitrary finite interval $[0,T]$ and obtain a solution in $ C([0,T],H^s)\times  C([0,T],H^{s+1}) $. 
\end{remark}

\begin{remark}
\label{rem_2_2}
Since $a_{12}(t,x,D_x)$ is a first order operator combining \eqref{eq:aux52} with \eqref{eq:aux51} we easily see that in order to get Sobolev well-posedness of order $s$ we need to take initial data $u_1^0$ and $u_2^0$ in $H^s$ and $H^{s+1}$, respectively, and right hand-side functions $f_1$ and $f_2$ in $C([0,T],H^s)$ and $C([0,T],H^{s+1})$, respectively.
\end{remark}

We have therefore proved the following theorem stated for the first time in \cite[Theorem 7.2]{GramRuz2013}.  
 \begin{thm} 
	Consider the Cauchy problem \eqref{eq:CP2x2} where $A(t,x,D_x)$ and $B(t,x,D_x)$ are $2\times 2$ matrices of first order and zero order pseudo-differential operators continuous with respect to $t$, respectively. Assume that ${b}_{21} \in  C([0,T],\Psi^{-1}_{1,0})$, the right hand-side functions $f_1$ and $f_2$ belong to $C([0,T],H^s)$ and $C([0,T],H^{s+1})$, respectively, and the initial data $u_1^0$ and $u_2^0$ belong to $H^s$ and $H^{s+1}$, respectively. Then, \eqref{eq:CP2x2} has a unique solution in $C([0,T],H^s)\times C([0,T],H^{s+1}) $. More generally  it is well-posed in the anisotropic Sobolev space $C([0,T],H^{s_1})\times C([0,T],H^{s_2})$, provided $s_2-s_1 = 1$. 
\end{thm}

\begin{remark}
\label{rem_GR}
It was also shown in \cite{GramRuz2013} that the solution $u$ satisfies the estimate
\begin{equation*}
	\|u_1(t,\cdot)\|_{H^s} + \|u_2(t,\cdot)\|_{H^{s+1}} \leq c e^{ct} \left( \|u_1^0\|_{H^s} + \|u_2^0\|_{H^{s+1}} \right), \quad 0 \leq t \leq T,
	\end{equation*} for $u_j^0 \in H^{s+j-1}$, $j=1,2$ with $c>0$ depending on $s$, $T$, and the support of the initial data.	
Since well-posedness is obtained for any Sobolev order $s$ it follows that the Cauchy problem  \eqref{eq:CP2x2} is also $C^\infty$ well-posed.	
\end{remark}

\subsection{The case $m=3$}

In this section we will extend the construction to the case of $3\times 3$ systems. In the argument there is an additional substitution and a fixed point argument step compared to the case $m=2$. The advantage of giving the case of $m=3$ here is that we can make the argument more concrete compared to the more abstract construction in the general case that will be given in the following section. Thus, let  
\begin{eqnarray} \label{eq:CP3x3}
\left\{ \begin{array}{l}
D_t u = A(t,x,D_x)u + B(t,x,D_x)u + f(t,x), \quad (t,x) \in [0,T] \times \R^n,\\
\left.u\right|_{t=0} = u_0, \quad x \in \R^n,
\end{array}
\right.
\end{eqnarray} 
where $u_0(x) = \big[u_1^0(x),u_2^0(x),u_3^0(x) \big]^T$, $f(t,x) = \big[ f_1(t,x) ,f_2(t,x),f_3(t,x) \big]^T$, $A(t,x,D_x)$ is defined by the matrix
\begin{equation} \label{eq:A3x3}
\begin{bmatrix}
\lambda_1(t,x,D_x) & a_{12}(t,x,D_x) & a_{13}(t,x,D_x) \\
0 & \lambda_2(t,x,D_x) & a_{23}(t,x,D_x) \\
0 & 0 & \lambda_3(t,x,D_x)
\end{bmatrix},
\end{equation} and
\[
B(t,x,D_x) =
\begin{bmatrix}
b_{11}(t,x,D_x) & b_{12}(t,x,D_x) & b_{13}(t,x,D_x) \\
b_{21}(t,x,D_x) & b_{22}(t,x,D_x) & b_{23}(t,x,D_x) \\
b_{31}(t,x,D_x) & b_{32}(t,x,D_x) & b_{33}(t,x,D_x)
\end{bmatrix}.
\]
We assume that all the entries of $A(t,x,D_x)$ and $B(t,x,\xi)$  belong to $ C \Psi_{1,0}^1$ and $ C \Psi_{1,0}^0$, respectively. Using the notations introduced earlier, we can write \begin{equation} \label{eq:u321}
	\begin{aligned}
	 u_3(t,x) &= U^0_3 + G_3(b_{31}u_1) + G_3(b_{32}u_2), \\
	 u_2(t,x) &= U^0_2 + G_2((a_{23}+b_{23})u_3) + G_2(b_{21}u_1), \\
	 u_1(t,x) &= U^0_1 + G_1((a_{12}+b_{12})u_2) + G_1((a_{13}+b_{13})u_3),
	\end{aligned}
	\end{equation} where \begin{equation} \label{eq:InitCondAux}
U_j^0(t,x) = G_j^0( u_j^0) + G_j(f_j), \quad j=1,2,3.
\end{equation} Now, we plug $u_3$ into $u_1$ and $u_2$ in formula \eqref{eq:u321} and, thus, obtain \begin{equation} \label{eq:aux53} \begin{aligned}
u_2(t,x) &= \widetilde{U}_2^0 + G_2(b_{21}u_1) + G_2((a_{23}+b_{23})G_3(b_{31}u_1)) \\
&\qquad + G_2((a_{23}+b_{23})G_3(b_{32}u_2)), \\
u_1(t,x) &= \widetilde{U}_1^0 + G_1((a_{13}+b_{13})G_3(b_{31}u_1)) +  \\
& \quad + G_1((a_{13}+b_{13})G_3(b_{32}u_2)) + G_1((a_{12}+b_{12})u_2),
\end{aligned} \end{equation} where \begin{equation*}
\widetilde{U}_j^0 = U^0_j + G_j((a_{j3}+b_{j3})(t,x,D_x)U^0_3), \quad j=1,2.
 \end{equation*}

We introduce the operator $\mathcal{G}^0_2$ by setting
\begin{equation}
\label{G2}
\mathcal{G}^0_2u_2:=G_2((a_{23}+b_{23})(t,x,D_x)G_3(b_{32}(t,x,D_x)u_2))
\end{equation}
and in analogy with the case $m=2$ we define
\begin{equation*}
	L_2 u_2 := u_2 - \mathcal{G}^0_2 u_2.  
\end{equation*} 
By Lemma \ref{thm:FIOest} we have that for any $s$, $\mathcal{G}^0_2$ has the operator norm in $H^s$ strictly less than 1 on a sufficiently small interval $[0,T^\ast]$, so $L_2$ is a perturbation of the identity operator. By the Neumann series it follows that $L_2$ is invertible as a continuous operator from $ C([0,T^\ast],H^s)$ to $ C([0,T^\ast],H^s)$. Noting now that 
\[
u_2-\mathcal{G}^0_2u_2=L_2u_2=\widetilde{U}_2^0 + G_2(b_{21}u_1) + G_2((a_{23}+b_{23})G_3(b_{31}u_1)), 
\]
we have that 
\begin{eqnarray*}
	u_2(t,x) = L^{-1}_2\widetilde{U}_2^0 + L^{-1}_2G_2((a_{23}+b_{23})G_3(b_{31}u_1)) + L^{-1}_2G_2(b_{21}u_1). 
\end{eqnarray*} Since this expression depends only on $u_1$, we can plug it into  the formula for $u_1$ in \eqref{eq:aux53} and obtain 
\begin{equation*}
\begin{aligned}
u_{1}(t,x) &= \widetilde{U}_1^0 + G_1((a_{13}+b_{13})G_3(b_{31}u_1)) + \\
& \qquad + G_1((a_{13}+b_{13})G_3(b_{32}u_2)) + G_1((a_{12}+b_{12})u_2)\\
&= \widetilde{U}_1^0 + G_1((a_{13}+b_{13})G_3(b_{31}u_1)) +  \\
& \qquad + G_1((a_{13}+b_{13})G_3(b_{32}(L^{-1}_2\widetilde{U}_2^0)) \\
& \qquad + G_1((a_{13}+b_{13})G_3(b_{32} L^{-1}_2G_2((a_{23}+b_{23})G_3(b_{31}u_1))))\\
& \qquad + G_1((a_{13}+b_{13})G_3(b_{32}(L^{-1}_2G_2(b_{21}u_1)))\\
& \qquad + G_1((a_{12}+b_{12})L^{-1}_2\widetilde{U}_2^0)\\
& \qquad + G_1((a_{12}+b_{12})  L^{-1}_2G_2((a_{23}+b_{23})G_3(b_{31}u_1)))\\
& \qquad + G_1((a_{12}+b_{12}) L^{-1}_2G_2(b_{21}u_1)).
\end{aligned}
\end{equation*} 
By collecting now the terms with order $\le 0$ we can simplify the previous formula as follows:
\begin{equation*}
\begin{aligned}
u_{1}(t,x) 
&= \widetilde{U}_1^0 + G_1(a_{13}G_3(b_{31}u_1)) + G_1(a_{13}G_3(b_{32}(L^{-1}_2\widetilde{U}_2^0)))\\
& \qquad + G_1(a_{13}G_3b_{32} L^{-1}_2G_2(a_{23}G_3(b_{31}u_1)))\\
& \qquad + G_1(a_{13}G_3b_{32} L^{-1}_2G_2(b_{23}G_3(b_{31}u_1)))\\
& \qquad + G_1(a_{13}G_3b_{32}(L^{-1}_2G_2(b_{21}u_1)))\\
& \qquad + G_1(b_{13}G_3(b_{32} L^{-1}_2G_2(a_{23}G_3(b_{31}u_1))))\\
& \qquad + G_1(a_{12}L^{-1}_2\widetilde{U}_2^0)\\
& \qquad + G_1(a_{12}L^{-1}_2G_2(a_{23}G_3(b_{31}u_1)))\\
& \qquad + G_1(a_{12}L^{-1}_2G_2(b_{23})G_3(b_{31}u_1)))\\
& \qquad + G_1(b_{12}L^{-1}_2G_2((a_{23}G_3(b_{31}u_1)))\\
& \qquad + G_1(a_{12}L^{-1}_2G_2(b_{21}u_1)) + \text{l.o.t}.
\end{aligned}
\end{equation*} 
Looking at the terms
\[
\begin{aligned}
&G_1(a_{13}G_3(b_{32}(L^{-1}_2\widetilde{U}_2^0))),\\
&G_1(a_{12}L^{-1}_2G_2(b_{21}u_1)),\\
&G_1(a_{12}L^{-1}_2G_2(a_{23}G_3(b_{31}u_1)))
\end{aligned}
\]
and keeping in mind that in order to get the right Sobolev regularity we need to have operators of order $0$, we deduce that $b_{21}$ and $b_{32}$ must have order $-1$ while $b_{31}$ must have order $-2$. Considering now the initial data 
\begin{equation*}
\widetilde{U}_j^0 = U^0_j + G_j((a_{j3}+b_{j3})(t,x,D_x)U^0_3), \quad j=1,2,
 \end{equation*}
by using \eqref{eq:InitCondAux} we obtain
\[
\widetilde{U}_j^0 = U^0_j + G_j((a_{j3}+b_{j3})(G^0_3(u^0_3)+G_3(f_3))), \quad j=1,2.
\]
Combining these formulas with an analysis of the term $G_1(a_{12}L_2^{-1}\widetilde{U}_2^0)$ we deduce that $\widetilde{U}_2^0$ must belong to $H^{s+1}$. This implies $U_2^0\in H^{s+1}$ and $U_3^0\in H^{s+2}$. Concluding, 
similarly to the case $m=2$, that is by the Banach fixed point theorem argument on $u_1$ and substitution in $u_2$ and $u_3$, we get anisotropic Sobolev well-posedness by assuming $u^0_1$ and $f_1$ in $H^s$, $u^0_2$ and $f_2$ in $H^{s+1}$, and $u^0_3$ and $f_3$ in $H^{s+2}$. This well-posedness is obtained by means of one invertible operator $L_2$, and in analogy with case $m=2$ the well-posedness can be extended to the whole interval $[0,T]$ by an iterated argument.
This proves Theorem \ref{main_theo_mi} in the case $m=3$.

\subsection{The general case}

We are now ready to prove the main result of our paper in the general case of an upper-triangular $m\times m$ matrix, i.e, a matrix $A$ of the type
\[
\begin{bmatrix}
\lambda_1(t,x,D_x) & a_{12}(t,x,D_x) &  \cdots & a_{1m}(t,x,D_x) \\
0 & \lambda_2(t,x,D_x) & \cdots & a_{2m}(t,x,D_x)\\
\vdots & \vdots & \vdots & \vdots\\
0 & 0 &  \lambda_{m-1}(t,x,D_x) & a_{m-1 m}(t,x,D_x)\\
0 & 0 &  \cdots & \lambda_m(t,x,D_x)
\end{bmatrix}.
\]
For the convenience of the reader we recall here the statement of Theorem \ref{main_theo_mi}.

\setcounter{section}{1}
\setcounter{thm}{0}
\begin{thm} 
Let 
\begin{equation}
\label{eq_CP_m}
\left\{\begin{array}{l}
D_t u = A(t,x,D_x)u + B(t,x,D_x)u + f(t,x), \quad (t,x) \in [0,T] \times \R^n, \\
\left.u \right|_{t=0} = u_0(x), \quad x \in \R^n,
\end{array} \right.
\end{equation} where $A(t,x,D_x) = [a_{ij}(t,x,D_x)]_{i,j=1}^m$ is an upper-triangular matrix of pseudo-diffe\-ren\-tial operators of order $1$, and $B(t,x,D_x) = [b_{ij}(t,x,D_x)]_{i,j=1}^m$  is a matrix of pseudo-differential operators of order $0$,  continuous with respect to $t$. 
Hence, if the lower order terms $b_{ij}$ belong to $C([0,T], \Psi^{j-i})$ for $i> j$, $u^0_k\in H^{s+k-1}$ and $f_k\in C([0,T],H^{s+k-1})$ for $k=1,\dots,m$ then \eqref{eq_CP_m} has a unique anisotropic Sobolev solution $u$, i.e.,  $u_k\in C([0,T], H^{s+k-1})$ for $k=1,\dots, m$.
\end{thm}
\setcounter{section}{2}
\setcounter{thm}{5}

\begin{proof}
Making use of the notations introduced earlier we can write the components of the solution $u$ as
\begin{equation}
\label{formula_u_m}
\begin{split}
u_i(t,x) &= U_i^0 + G_i\left( \sum_{j>i}^m a_{ij}(t,x,D_x)u_j \right) + G_i\left( \sum_{\substack{j=1 \\ j \neq i}}^m b_{ij}(t,x,D_x)u_j \right)\\
&= U_i^0+\sum_{j<i} G_i (b_{ij}(t,x,D_x)u_j)+\sum_{i<j\le m} G_i((a_{ij}+b_{ij})(t,x,D_x)u_j),
\end{split} 
\end{equation}
where  
\[
U_i^0 = G_i^0 u_j^0 + G_i(f_i), 
\]
and $G_i, G_i^0$ are Fourier integral operator of order $0$ for $i=1,\dots, m$. Note that from the fact that $b_{ij}$ is a symbol of order $0$ for every $i,j$ and, in particular, of order $j-i$ for $j<i$ we obtain that the operator $G_i(b_{ij})$ is of order $j-i$ for $j<i$, while $G_i(a_{ij}+b_{ij})$ is, in general, of order $1$. To simplify the argument we introduce the notations $G^{j-i}_{ij}$ and $G^1_{ij}$ for the operators $G_i(b_{ij})$ and $G_i(a_{ij}+b_{ij})$, respectively. Here the superscript stands to remind us of the order of the operator. Hence,
\[
u_i=U_i^0+\sum_{j<i} G^{j-i}_{ij}(u_j)+\sum_{i<j\le m} G^1_{ij}(u_j),
\]
for $i=1,\dots,m$. By begin by substituting  
\[
u_m=U_m^0+\sum_{j<m} G^{j-m}_{mj}(u_j),
\]
into
\[
u_{m-1}=U_{m-1}^0+\sum_{j<m-1} G^{j-m+1}_{m-1,j}(u_j)+G^1_{m-1,m}(u_m).
\]
We get
\[
\begin{split}
u_{m-1}&=U_{m-1}^0+\sum_{j<m-1} G^{j-m+1}_{m-1,j}(u_j)+G^1_{m-1,m}U_m^0+\sum_{j<m} G^1_{m-1,m}G^{j-m}_{mj}(u_j)\\
&=(U_{m-1}^0+G^1_{m-1,m}U_m^0)+\sum_{j<m-1} (G^{j-m+1}_{m-1,j}(u_j)+G^1_{m-1,m}G^{j-m}_{mj}(u_j))\\
&+ G^1_{m-1,m}G^{-1}_{m,m-1}u_{m-1}.
\end{split}
\]
Note that it is enough to assume $U^0_m\in H^{s+1}$ and $U^0_{m-1}\in H^s$ to obtain $U_{m-1}^0+G^1_{m-1,m}U_m^0\in H^{s}$. Since all the operators above are of order $\le 0$ we conclude that the operator
\[
L_{m-1}=I- G^1_{m-1,m}G^{-1}_{m,m-1}:=I-\mathcal{G}^0_{m-1}
\]
is invertible on a sufficiently small interval $[0,T]$ and, therefore, 
\begin{multline}
\label{u_m-1}
u_{m-1}- G^1_{m-1,m}G^{-1}_{m,m-1}u_{m-1}=(U_{m-1}^0+G^1_{m-1,m}U_m^0)\\+\sum_{j<m-1} (G^{j-m+1}_{m-1,j}(u_j)+G^1_{m-1,m}G^{j-m}_{mj}(u_j)),
\end{multline}
yields
\begin{equation}
\label{L_m-1}
u_{m-1}=L^{-1}_{m-1}\widetilde{U}^0_{m-1}+L^{-1}_{m-1}\sum_{j<m-1}\widetilde{G}^{j-m+1}_{m-1}u_j,
\end{equation}
with $\widetilde{U}^0_{m-1}$ and $\widetilde{G}^{j-m+1}_{m-1}$ defined by the right-hand side of \eqref{u_m-1}. We now substitute $u_{m}$ and $u_{m-1}$ into $u_{m-2}$ making use of \eqref{L_m-1}. We obtain
\begin{equation}
\label{u_m-2_form}
\begin{split}
&u_{m-2}=U_{m-2}^0+\sum_{j<{m-2}} G^{j-m+2}_{m-2,j}(u_j)+G^1_{m-2,m-1}(u_{m-1})+G^1_{m-2,m}(u_{m})\\
&=U_{m-2}^0+\sum_{j<{m-2}} G^{j-m+2}_{m-2,j}(u_j)+G^1_{m-2,m-1}L^{-1}_{m-1}\widetilde{U}^0_{m-1}\\
&+G^1_{m-2,m-1}L^{-1}_{m-1}\sum_{j< m-2}\widetilde{G}^{j-m+1}_{m-1}u_j+ G^1_{m-2,m-1}L^{-1}_{m-1} \widetilde{G}^{-1}_{m-1}u_{m-2}\\
&+G^1_{m-2,m}U^0_m+ G^1_{m-2,m}\sum_{j<m-2} G^{j-m}_{mj}(u_j)+G^1_{m-2,m}G^{-2}_{m,m-2}u_{m-2}\\
&+G^1_{m-2,m}G^{-1}_{m,m-1}L^{-1}_{m-1}\widetilde{U}^0_{m-1} \\
&+G^1_{m-2,m}G^{-1}_{m,m-1}L^{-1}_{m-1}\sum_{j<m-2}\widetilde{G}^{j-m+1}_{m-1}u_j+G^1_{m-2,m}G^{-1}_{m,m-1}L^{-1}_{m-1} \widetilde{G}^{-1}_{m-1}u_{m-2}.
\end{split}
\end{equation}
We set 
\begin{multline}
\label{indata_m-2}
\widetilde{U}^0_{m-2}=U_{m-2}^0+G^1_{m-2,m-1}L^{-1}_{m-1}\widetilde{U}^0_{m-1}\\
+G^1_{m-2,m}U^0_m+G^1_{m-2,m}G^{-1}_{m,m-1}L^{-1}_{m-1}\widetilde{U}^0_{m-1}. 
\end{multline}
The operators $G^1_{m-2,m-1}L^{-1}_{m-1}$ and $G^1_{m-2,m}$ in \eqref{indata_m-2} are of order 1. Keeping in mind that we already assumed $U^0_{m}\in H^{s+1}$ and $U^0_{m-1}\in H^s$, in order to obtain Sobolev order $s$ the initial data $U^0_m$, $U^{0}_{m-1}$ and $U^0_{m-2}$ must belong to $H^{s+2}$, $H^{s+1}$ and $H^s$, respectively.  Thus,
\begin{equation}
\label{u_m-2}
u_{m-2}= \widetilde{U}^0_{m-2}+\mathcal{G}^0_{m-2}u_{m-2}+\sum_{j<m-2}\widetilde{G}^{j-m+2}_{m-2}u_j,
\end{equation}
where $\mathcal{G}^0_{m-2}$ is a zero order operator defined by 
\begin{multline*}
\mathcal{G}^0_{m-2}u_{m-2}=G^1_{m-2,m-1}L^{-1}_{m-1} \widetilde{G}^{-1}_{m-1}u_{m-2}+G^1_{m-2,m}G^{-2}_{m,m-2}u_{m-2}\\
+G^1_{m-2,m}G^{-1}_{m,m-1}L^{-1}_{m-1} \widetilde{G}^{-1}_{m-1}u_{m-2},
\end{multline*}
and the last summand in \eqref{u_m-2} is obtained by collecting all the operators acting on $u_j$ with $j<m-2$ in \eqref{u_m-2_form}. Since the norm of $\mathcal{G}^0_{m-2}$ can be taken strictly less than one in a sufficiently small interval $[0,T]$ we have that the operator 
\[
L_{m-2}=I-\mathcal{G}^0_{m-2}
\]
is invertible and, therefore,
\begin{equation}
\label{L_m-2}
u_{m-2}=L_{m-2}^{-1} \widetilde{U}^0_{m-2} +\sum_{j<m-2}L_{m-2}^{-1}\widetilde{G}^{j-m+2}_{m-2}u_j.
\end{equation}
Note that $ \widetilde{U}^0_{m-2} \in H^s$ if $U^0_{m}\in H^{s+2}$, $U^0_{m-1}\in H^{s+1}$ and $U^0_{m-2}\in H^s$. By iterating the same procedure we deduce that
\begin{equation}
\label{u_k}
u_{k}= \widetilde{U}^0_{k}+\mathcal{G}^0_{k}u_k+\sum_{j<k}\widetilde{G}^{j-k}_{k}u_j,
\end{equation}
where $\widetilde{U}^0_k$ depends on $U^0_k$, $U^0_j$ and $\widetilde{U}^0_j$ with $j>k$ and $\mathcal{G}^0_k$ is a zero order operator defined by using invertible operators $L_{m-1}$, $L_{m-2}$,\dots, $L_{k}$. In addition, we obtain $\widetilde{U}^0_k\in H^s$ since $U^0_m\in H^{s+m-k}$, $U^0_{m-1}\in H^{s+m-k-1},\dots, U^0_k\in H^s$. It follows that for $k=2$ we have 
\[
u_2=\widetilde{U}^0_{2}+\mathcal{G}^0_{2}u_2+\widetilde{G}^{-1}_{2}u_1,
\]
where $\mathcal{G}^0_2$ is a zero order operator defined by invertible operators $L_{m-1}$, $L_{m-2}$,\dots, $L_{2}$, $\widetilde{G}^{-1}_{2}$ is of order $-1$, and $\widetilde{U}^0_2\in H^s$ since $U^0_m\in H^{s+m-2}$, $U^0_{m-1}\in H^{s+m-3},\dots, U^0_2\in H^s$. Hence, by inverting the operator $L_2=I-\mathcal{G}^0_2$ on a sufficiently small interval $[0,T]$ we have
\[
u_2=L^{-1}_2\widetilde{U}^0_{2}+L^{-1}_2\widetilde{G}^{-1}_{2}u_1.
\]
Now by substitution of $u_2, u_3, \dots, u_m$ in the equation of $u_1$ we arrive at the formula \eqref{u_k} with $k=1$, i.e.,
\[
u_1= \widetilde{U}^0_{1}+\mathcal{G}^0_{1}u_1,
\]
where $\widetilde{U}^0_{1}\in H^s$ since $U^0_m\in H^{s+m-1}$, $U^0_{m-1}\in H^{s+m-2},\dots, U^0_2\in H^{s+1}, U^0_1\in H^s$. Concluding, by the Banach fix point argument we prove that there exists a unique $u_1\in C([0,T], H^s)$ solving the equation above with the given initial conditions.  By substitution in the equations for $u_2, \dots, u_{m-1}, u_m$ we arrive at the desired Sobolev well-posedness with $u_k\in C([0,T], H^{s+k-1})$ for $k=2,\dots, m$. Note that, since the sufficiently small interval $[0,T]$ where we get well-posedness does not depend on the initial data, by a standard iteration argument we can achieve well-posedness on any bounded interval $[0,T]$ as stated in the theorem.
 \end{proof}

\section{Schur decomposition of $m \times m$ matrices}
\label{SEC:Schur}

In this section we investigate how to reduce an $m\times m$ matrix to the upper triangular form. 
We recall that such decomposition is well-known for constant matrices and goes under the name of Schur's triangularisation, with its statement given in Theorem \ref{THM:Schuri}.

One of the difficulties when dealing with variable multiplicities is the loss of regularity in the parameters at the points of multiplicities. In the following, we will assume that $A$ is a matrix of first order symbols continuous with respect to $t$, i.e., $A(t,x,\xi) \in \big( C S^1 \big)^{m \times m}$.

We will now develop a parameter dependent extension of the Schur triangularisation procedure and we will describe it step by step. Then we will give an example for it for the systems of low sizes, namely, for $m=2$ and $m=3$.

In the case of $m=2$ the construction below was introduced in  \cite{GramRuz2013} and now we give its general version for systems of any size.

\subsection{First step or Schur step}

The first step in our triangularisation follows the construction in the constant case except that we will not get a unitary transformation matrix.  For this reason we talk of a Schur step. Throughout this paper $e_i$ denotes the $i$-th vector of the standard basis of $\R^n$ with an appropriate dimension $n$.

\begin{proposition}[Schur step] \label{prop:PrepSchur} \hfill \\ Let the $m \times m$ matrix valued symbol $A(t,x,\xi)=[a_{ij}(t,x,\xi)]_{i,j=1}^{m}$, $a_{ij} \in C S^1$, have an eigenvalue $\lambda \in C S^1$ and a corresponding eigenvector $h \in \big( C S^1 \big)^{m}$ such that there exists $j \in \{ 1, \dots, m \}$ with \begin{equation} \label{eq:EVneq0}
		\la h(t,x,\xi) | e_j \ra \neq 0 \quad \forall (t,x,\xi) \in [0,T] \times \mathbb{R}^n \times \{ |\xi| \geq M \},
	\end{equation} for a sufficiently large $M>0$. Then there exist an $m \times m$ matrix valued symbol $T(t,x,\xi)$, invertible for $(t,x,\xi) \in [0,T] \times \mathbb{R}^n \times \{ |\xi| \geq M \}$, and an $(m-1) \times (m-1)$ matrix  valued symbol $E(t,x,\xi)$ with entries in  $C S^0$ and $C S^1$, respectively, such that \begin{equation*}
	T^{-1}(t,x,\xi)A(t,x,\xi)T(t,x,\xi) = \begin{bmatrix}
		\lambda   & a_{12}  & \cdots & a_{1m} \\ 
		0         &         &        &         \\
		\vdots    &         &   E(t,x,\xi)    &         \\
		0         &         &        &         \end{bmatrix}
\end{equation*} for all $(t,x,\xi) \in [0,T] \times \mathbb{R}^n \times \{ |\xi| \geq M \}$.
\end{proposition}

\begin{proof} First let us note that we can assume that $j=1$ in \eqref{eq:EVneq0}. If that is not the case, we can exchange the rows $1$ and $j$ as well as columns $1$ and $j$ to move the $j$th component of the eigenvector to the first component. 
	
	We define the rescaled eigenvector $\mu$ componentwise by
	\begin{equation*}
		\mu_{i}(t,x,\xi) = \frac{\la h(t,x,\xi) | e_i \ra}{ \la h(t,x,\xi) | e_1 \ra} \quad \forall i = 1, \dots, m.
	\end{equation*}
	
	
	Now we set \begin{equation*}
		T(t,x,\xi) = \begin{bmatrix}
			\mu_{1}  & 0   & \dots   & 0 \\
			\mu_{2}	&     &         &         \\
			\vdots    	&     & I_{m-1} &         \\
			\mu_{m}	&     &         &
		\end{bmatrix}.
	\end{equation*}
	Since $\mu_1\equiv 1$ it follows that 
	
	\begin{equation*}
		\quad T^{-1}(t,x,\xi) = \begin{bmatrix}
			\mu_{1}  & 0    & \dots   & 0 \\
			-\mu_{2}	&      &         &         \\
			\vdots    	&      & I_{m-1} &         \\
			-\mu_{m}	&      &         &
		\end{bmatrix},
	\end{equation*} where $I_{m-1}$ is the $(m-1)\times(m-1)$ identity matrix. By direct computations we get \begin{equation*}
	A T = \begin{bmatrix}
		\sum\limits_{j=1}^m a_{1j} \mu_{j} & & & \\
		\vdots & A_{(2)} & \dots & A_{(m)}\\
		\sum\limits_{j=1}^m a_{mj} \mu_{j} & & &
	\end{bmatrix} 
	= \begin{bmatrix}
		\lambda \mu_{1} & & & \\
		\vdots & A_{(2)} & \dots & A_{(m)}\\
		\lambda \mu_{m} & & &
	\end{bmatrix},
\end{equation*} where we used that 
\begin{equation} \label{eq:EV1}
\sum_{j=1}^m a_{ij} \mu_{j} = \lambda \mu_{i}, \quad i=1,\dots,m,
\end{equation} 
and denoted the $i$th column of $A$ by $A_{(i)}$. The equations in \eqref{eq:EV1} are given by the eigenvalue equation $A \mu = \lambda \mu$. Further, from $\mu_{1} \equiv 1$ we obtain \begin{eqnarray}
\nonumber T^{-1} A T &=& \begin{bmatrix}
	\lambda \mu_{1}^2 & a_{12}\mu_{1} & \dots & a_{1m}\mu_{1} \\
	-\mu_{2}\mu_{1} \lambda + \mu_{2} \lambda &  &  & \\
	\vdots & & E & \\
	-\mu_{m} \mu_{1} \lambda + \mu_{m} \lambda & & & &
\end{bmatrix} \\
\label{aux:Trafo1} &=& \begin{bmatrix}
	\lambda & a_{12} & \dots & a_{1m} \\
	0 &  &  & \\
	\vdots & & E & \\
	0 & & & &
\end{bmatrix},
\end{eqnarray} which concludes the proof. Note that by construction the matrix $E$ has entries in $C S^1$ which depend on $A$. In particular its eigenvalues are the eigenvalues of $A$ excluding $\lambda$ (counted as many times as they occur).  
\end{proof} 



Applying Proposition \eqref{prop:PrepSchur} repeatedly for $m-2$ times to $E$, we obtain a full Schur transformation of $A$, that is a full reduction to an upper triangular form. In the next subsection we describe this iteration in detail. This triangularisation procedure is summarised in Theorem \ref{thm:Schur} where sufficient conditions on the eigenvectors of $A$ are given.\\

\subsection{The triangularisation procedure} \label{sec:Procedure}

The reduction to an upper triangular form or the Schur transformation of $A$ is possible under certain conditions on its eigenvectors. More precisely, let 
$$h_1(t,x,\xi),\dots, h_{m-1}(t,x,\xi) \in \big( C S^0 \big)^m$$ be $m-1$ eigenvectors of $A(t,x,\xi) = [a_{ij}(t,x,\xi)]_{i,j=1}^m$,  $a_{ij} \in C S^1$, corresponding to the eigenvalues $\lambda_1(t,x,\xi)$, $\dots$, $\lambda_{m-1}(t,x,\xi) \in C S^1$. To formulate the sufficient conditions for the existence of such Schur transformation, we introduce a set of auxiliary vectors $h^{(i)}$, $i=1,\dots,m-1$, which depend only on $h_i$ and the previous vectors $h^{(j)} \in C S^0$, $j=1,\dots,i-1$. When $i=1$ we set $h^{(1)}=h_1$.

As in Proposition \ref{prop:PrepSchur} we begin by assuming 
\begin{equation}
	\label{cond_1}
	\la h^{(1)}(t,x,\xi) | e_1 \ra \neq 0
\end{equation}
for $(t,x,\xi) \in [0,T] \times \mathbb{R}^n \times \{ |\xi| \geq M \}$.

\begin{remark} \label{rem:Transform}
	As noted in the proof of Proposition \ref{prop:PrepSchur}, we could have that $$\la h^{(1)}(t,x,\xi) | e_j \ra \neq 0$$ for another arbitrary $j \in \{ 1, \dots,m \}$. Then, we could transform the matrix $A(t,x,\xi)$ by a constant permutation matrix $P$ such that $P^{-1}h^{(1)}$ is eigenvector of $P^{-1}AP$ corresponding to $\lambda_1$ which satisfies $\la P^{-1} h^{(1)}(t,x,\xi) | e_1 \ra \neq 0$. For this reason we state \eqref{cond_1} with $h^{(1)}$ and $e_1$.
\end{remark}

\begin{enumerate}[{\quad {\bf Step}} {\bf 1}]
	\item By Proposition \ref{prop:PrepSchur} there exists a matrix $T_1$ such that \begin{equation*}
		T^{-1}_1 A T_1 = \begin{bmatrix}
			\lambda_1 & a_{12} & \cdots & a_{1m} \\ 
			0         &         &        &         \\
			\vdots    &         &   E_{m-1}  &         \\
			0         &         &        &         \end{bmatrix}.
	\end{equation*} The matrix $T_1$ is given by \begin{equation*}
	T_1 = \begin{bmatrix}
		\omega_1 & e_2 & \dots & e_m 
	\end{bmatrix}, \quad \omega_1 = \begin{bmatrix}
	\omega_{11}  & \dots & \omega_{1m}
\end{bmatrix}^T
\end{equation*} with \begin{equation*}
\omega_{1j} = \frac{\la h^{(1)}(t,x,\xi) | e_j \ra }{\la h^{(1)}(t,x,\xi) | e_1 \ra}.
\end{equation*}
In the sequel we make use of the projector $\Pi_k : \R^m \to \R^{m-k}$,  $0 \leq k \leq m-1$, defined by
\begin{equation*}
	\Pi_k \begin{bmatrix}
		x_1 \\ \vdots \\ x_m
	\end{bmatrix} = \begin{bmatrix}
	x_{k+1} \\ \vdots \\ x_m
\end{bmatrix}.
\end{equation*}
Note that $\Pi_0$ is the identity map $I_m:\R^m\to \R^m$.

\item[\quad {\bf Step 2}] Since $h_2$ is an eigenvector of $A$ with eigenvalue $\lambda_2$ we get that $T_1^{-1}h_2$ is an eigenvector of $T^{-1}_1A T_1$ with eigenvalue $\lambda_2$ as well. By the structure of $T^{-1}_1A T_1$ we easily see that $h^{(2)} := \Pi_1 T_1^{-1} h_2$ is an eigenvector of $E_{m-1}$, corresponding to $\lambda_2$.

Arguing as in Remark \ref{rem:Transform} we assume that \begin{equation}
	\label{cond_2}
	\la \Pi_1 T_1^{-1} h_2 | e_1 \ra \neq 0 \quad \forall (t,x,\xi) \in [0,T] \times \mathbb{R}^n \times \{ |\xi| \geq M \},
\end{equation} to be able to apply Proposition \ref{prop:PrepSchur} to $E_{m-1}$. We get that there exists an $(m-1) \times (m-1)$ matrix $\tilde{T}_2$ such that $\tilde{T}_2^{-1}E_{m-1}\tilde{T}_2$ is of form \begin{equation*}
\begin{bmatrix}
	\lambda_2 & \ast & \dots   & \ast \\
	0         &      &         &      \\
	\vdots    &      & E_{m-2} & \\
	0         &      &         &
\end{bmatrix},
\end{equation*} 
where in the first row the first row of $E_{m-1}$ appears. Thus, setting \begin{equation*}
	T_2 = \begin{bmatrix}
		1       &  0  &  \dots    &  0      \\
		0       &  	  &		 &		  \\
		\vdots  &     & \tilde{T}_2 &        \\
		0       &     &      &       
	\end{bmatrix},
\end{equation*} we obtain \begin{equation}
\label{T1T2}
T_2^{-1}T_1^{-1} A T_1 T_2 = \begin{bmatrix}
	\lambda_1 & \ast & \ast & \dots & \ast \\
	0 & \lambda_2 & \ast & \dots & \ast \\
	0 & 0 &  & & \\
	\vdots & \vdots &  & E_{m-2} & \\
	0 & 0 & & &
\end{bmatrix}.
\end{equation} 
Note that in \eqref{T1T2} we write explicitly only the entries most relevant to our triangularisation.
To compute the matrix $\tilde{T}_2$, we set \begin{equation*}
	\omega_2 = \begin{bmatrix}
		\omega_{22} & \dots & \omega_{2m}
	\end{bmatrix}^T, 
\end{equation*} where
 
\begin{equation*}
	\omega_{2j}(t,x,\xi) := \frac{\la h^{(2)}(t,x,\xi) | e_j \ra}{\la h^{(2)}(t,x,\xi) | e_1  \ra} , \quad j=2,\dots,m,
\end{equation*} and then \begin{equation*}
\tilde{T}_2 = \begin{bmatrix}
	\omega_2 & e_2 & \dots & e_{m-1}
\end{bmatrix}.
\end{equation*}
It is clear that $T_2$ has the same structure as $T_1$, i.e., it is defined via a rescaled eigenvector as the first column and an identity matrix ($I_{m-1}$ for $T_1$ and $I_{m-2}$ for $T_2$).
\item[\quad {\bf Step k}] By iterating the method $k-1$ times we can find $k-1$ matrices $T_1, T_2, \cdots, T_{k-1}$ of size  $m\times m$ such that  \begin{eqnarray*}
	&& T_{k-1}^{-1} \cdot \dots \cdot T_1^{-1} A T_1 \cdot \dots \cdot T_{k-1} = \\
	&& \qquad \begin{bmatrix}
		\lambda_1 & \ast & \ast & \dots & \dots & \ast \\
		0 & \ddots & \ast & \dots & \dots & \ast\\
		0 & 0 & \lambda_{k-1} & \ast & \dots &\ast\\
		0 & 0 & 0 & & & \\
		\vdots & \vdots & \vdots &  & E_{m-{k+1}}  & \\
		0 & 0 & 0 & &
	\end{bmatrix},
\end{eqnarray*} 
where $E_{m-{k-1}}$ is a $(m-k+1) \times (m-k+1)$ matrix and the equality is true on $[0,T]\times\mathbb{R}^n\times\{|\xi|\ge M\}$. Since $h_k$ is an eigenvector of $A$ corresponding to $\lambda_k$, the vector \begin{equation*}
	T_{k-1}^{-1} T_{k-2}^{-1} \cdot \dots \cdot T_1^{-1} h_k
\end{equation*} is an eigenvector of \begin{equation*}
T_{k-1}^{-1} T_{k-2}^{-1} \cdot \dots \cdot T_1^{-1} A T_1 T_2 \cdot \dots \cdot T_{k-1}
\end{equation*} and  \begin{equation*}
h^{(k)} := \Pi_{k-1} T_{k-1}^{-1} T_{k-2}^{-1} \cdot \dots \cdot T_1^{-1} h_k \in \big( C S^0 \big)^{m-k+1}
\end{equation*} an eigenvector of $E_{m-(k-1)}$ corresponding to $\lambda_k$. Thus, to satisfy the assumptions of Proposition \ref{prop:PrepSchur} and keeping in mind Remark \ref{rem:Transform}, we require that \begin{equation}
\label{cond_k}
\la h^{(k)}(t,x,\xi) | e_1 \ra \neq 0 \quad \forall (t,x,\xi) \in [0,T] \times \mathbb{R}^n \times \{ |\xi| \geq M \}.
\end{equation} It follows that there exists an $(m-k+1) \times (m-k+1)$ transformation matrix $\tilde{T}_k$ such that $\tilde{T}^{-1}_{k} \dots \tilde{T}_1^{-1}A\tilde{T}_{k} \dots \tilde{T}_1$ is of the form \begin{equation*}
\begin{bmatrix}
	\lambda_k & \ast & \dots   & \ast \\
	0         &      &         &      \\
	\vdots    &      & E_{m-k} & \\
	0         &      &         &
\end{bmatrix}.
\end{equation*} and set \begin{equation*}
T_k = \begin{bmatrix}
	I_{k-1} & \bf{0} \\ 
	\bf{0} & \tilde{T}_k
\end{bmatrix}.
\end{equation*} The matrix $\tilde{T}_k$ is defined by \begin{equation*}
\tilde{T}_k = \begin{bmatrix}
	\omega_k & e_2 & \dots & e_{m-k+1}
\end{bmatrix}, \quad \omega_k = \begin{bmatrix}
\omega_{kk} & \dots & \omega_{km}
\end{bmatrix}^T,
\end{equation*} where \begin{equation*}
\omega_{kj} = \frac{\la h^{(k)}(t,x,\xi) | e_j \ra}{\la h^{(k)}(t,x,\xi) | e_1 \ra}, \quad j = k, \dots, m.
\end{equation*}

\item[{\quad {\bf Step}} {\bf m-1}] This is the last step as $E_2$ is a $2 \times 2$ matrix. We have that \begin{equation*}
	h^{(m-1)} = \Pi_{m-2} T_{m-2}^{-1} \cdot \dots \cdot T_{1}^{-1} h_{m-1} \in \big( C S^0 \big)^{2}
\end{equation*} is an eigenvector of $E_2$ corresponding to $\lambda_{m-1}$ and that $\tilde {T}_{m-1}$ exists as before if \begin{equation}
\label{cond_m-1}
\la h^{(m-1)}(t,x,\xi) | e_1 \ra \neq 0 \quad \in \forall (t,x,\xi) \in [0,T] \times \mathbb{R}^n \times \{ |\xi| \geq M \}.
\end{equation} The matrix  $\tilde{T}_{m-1}$ is given by \begin{equation*}
\tilde{T}_{m-1} = \begin{bmatrix}
	\omega_{m-1} & e_{2}
\end{bmatrix} = \begin{bmatrix}
\omega_{m-1,m-1} & 0 \\
\omega_{m-1,m} & 1
\end{bmatrix},
\end{equation*} where \begin{equation*}
\omega_{m-1,j} = \frac{\la h^{(m-1)}(t,x,\xi) | e_j \ra }{\la h^{(m-1)}(t,x,\xi) | e_1 \ra} , \quad j =m-1,m,
\end{equation*} and then \begin{equation*}
T_{m-1} = \begin{bmatrix}
	I_{m-2} & \bf 0 \\
	\bf 0 & \tilde{T}_{m-1}
\end{bmatrix}.
\end{equation*}
\end{enumerate}

We are now ready to state Theorem \ref{thm:Schur} which summarises the triangularisation procedure explained above. For the convenience of the reader we recall the notations introduced so far:
\begin{itemize}
	\item $h_1, \dots, h_{m-1}$ are the eigenvectors of the matrix $A$ corresponding  to the eigenvalues $\lambda_1, \dots, \lambda_{m-1}$.
	\item $h^{(1)}=h_1$ and 
	\begin{equation}\label{EQ:his}
		h^{(i)} = \Pi_{i-1}T_{i-1}^{-1} T_{i-2}^{-1} \cdot \,\dots\, \cdot T_1^{-1} h_i \in \big( C S^0\big)^{m-k+1} ,
	\end{equation} 
	for $i=2, \dots,m-1$.
	\item the matrices $T_k$ are inductively defined as follows: $T_0 = I_{m}$ and \begin{equation*}
		T_k = \begin{bmatrix}
			I_{k-1} & \bf 0 \\
			\bf 0 & \tilde{T}_k
		\end{bmatrix}, \quad \tilde{T}_k = \begin{bmatrix}
		\omega_{k} & e_{2} & \dots &e_{m-k}
	\end{bmatrix}, \quad e_i \in \R^{m-k},
\end{equation*} where \begin{equation*}
\omega_{kj} = \frac{\la h^{(k)}(t,x,\xi) |  e_j \ra}{\la h^{(k)}(t,x,\xi), e_1 \ra}, \quad j = k,\dots,m.
\end{equation*} 
\end{itemize}
Finally, we note that $h^{(k)}$ depends only on $T_{k-1}$, $\dots$, $T_1$ and, thus, only on the eigenvectors $h^{(k-1)}$, $\dots$, $h^{(1)}$.


Summarising, we can formulate a more precise version of Theorem \ref{thm:Schuri}.

\begin{thm}[Schur Decomposition] \label{thm:Schur} \hfill \\
	Let $A(t,x,\xi) = [a_{ij}]_{i,j=1}^m$, $a_{ij} \in C S^1$ be a matrix with eigenvalues $\lambda_1, \dots,\lambda_{m} \in C S^1$, and let $h_1 , \dots, h_{m-1}  \in \big( C S^0 \big)^m$ be the corresponding eigenvectors. Suppose that for $e_1 \in \R^{m-i+1}$ the condition \begin{equation} \label{eq:CondThmSchur}
		\la h^{(i)}(t,x,\xi) | e_1 \ra \neq 0, \quad \forall (t,x,\xi) \in [0,T] \times \R^n \times \R^n 
	\end{equation} holds for all $i=1,\dots, m-1$, with the notation explained above. Then, there exists a matrix-valued symbol $T(t,x,\xi)= [t_{ij}]_{i,j=1}^m$, $t_{ij} \in C S^0$, invertible for $(t,x,\xi) \in [0,T] \times \R^n \times \{ |\xi| \geq M\}$, such that \begin{equation*}
	T^{-1}(t,x,\xi)A(t,x,\xi) T(t,x,\xi) = \Lambda(t,x,\xi) + N(t,x,\xi)
	\end{equation*} for all $(t,x,\xi) \in [0,T] \times \R^n \times \{ |\xi| \geq M \}$, where   \begin{equation*}
	\Lambda(t,x,\xi) = \diag(\lambda_1(t,x,\xi),\lambda_2(t,x,\xi),\dots,\lambda_m(t,x,\xi))
	\end{equation*} and \begin{equation*}
	N(t,x,\xi) = \begin{bmatrix}
	0 & N_{12}(t,x,\xi) & N_{13}(t,x,\xi) & \cdots & N_{1m}(t,x,\xi) \\
	0 & 0 & N_{23}(t,x,\xi) & \cdots & N_{2m}(t,x,\xi) \\
	\vdots & \vdots & \vdots & \cdots & \vdots \\
	0 & 0 & 0 & \dots & N_{m-1m}(t,x,\xi)\\
	0 & 0 & 0 & \dots & 0
	\end{bmatrix},
	\end{equation*}
	and $N$ is a nilpotent matrix with entries in $C S^1$. Furthermore, the matrix symbol $T$ is given by \begin{equation*}
	T(t,x,\xi) = T_1 T_2 \cdot \dots \cdot T_{m-1},
	\end{equation*} with the notation explained above.
\end{thm}

\begin{remark}
\label{rem_gen_cond}
	Taking into account Remark \ref{rem:Transform}, let us stress that condition \eqref{eq:CondThmSchur} is not restrictive as it can be replaced by the following: suppose that there exist $m-1$ numbers $j_i \in \{ 1,\dots, m-i+1 \}$, $i=1, \dots m-1$, such that for all $i=1, \dots, m-1$ \begin{equation} \label{eq:auxSchur2}
			\la h^{(i)}(t,x,\xi) | e_{j_i} \ra \neq 0 \quad \forall (t,x,\xi) \in [0,T] \times \R^n \times \{ |\xi| \geq M \}
		\end{equation} holds.
\end{remark}


\begin{remark}
	Theorem \ref{thm:Schur} is quite general in the sense that the functions $a_{ij}$ could be complex-valued. In this paper, we are concerned with hyperbolic matrices, i.e. we assume that the eigenvalues $\lambda_1, \dots, \lambda_m$ are real. We stress that the Schur transform does not change the hyperbolicity of the matrix as the eigenvalues of $T^{-1} A T$ are also $\lambda_1, \dots, \lambda_m$.
\end{remark}

\begin{remark}
	For our applications in this and future work it is important that the transform $T$ in Theorem \ref{thm:Schur} keeps the regularity of the original matrix $A$, i.e. that the elements of the Schur transform $T^{-1} A T$ are in the same class as the elements of $A$. Here, we stated everything with $C S^1$  and $C S^0$ as that is the regularity considered in this paper. Note that one could replace $C$ with $C^k$ or $C^\infty$ and find a matrix $T$ such that the transformed matrix $T^{-1}AT$ inherits the same regularity with respect to $t$. In addition, one could also drop the regularity in $t$ to $L^\infty$ and the triangularisation procedure would still work preserving the boundedness in $t$ through every step. 
\end{remark}
For the sake of simplicity and the reader's convenience, in the next subsections we analyse Theorem   \ref{thm:Schur} in the special cases of $m=2$ and $m=3$.  

\subsection{The case $m=2$}


We now formulate Theorem \ref{thm:Schur} in the special case $m=2$. In this way we recover the formulation given in \cite{GramRuz2013}.  

\begin{thm}[{\cite[Theorem 7.1]{GramRuz2013}}]
	Suppose that $A(t,x,\xi)=[a_{ij}(t,x,\xi)]_{i,j=1}^2$ admits eigenvalues $\lambda_j(t,x,\xi) \in C S^1$, $j=1,2$, and an eigenvector $h(t,x,\xi) \in (C S^0)^2$ satisfying \begin{equation} \label{cond:EVnonzero}
		\la h(t,x,\xi) | e_j \ra \neq 0, \quad (t,x,\xi) \in [0,T] \times \R^n \times \{|\xi| \geq M\},
	\end{equation} for $j=1$ or $j=2$. Then, we can find a $2 \times 2$ matrix valued symbol $T(t,x,\xi)=[t_{ij}(t,x,\xi)]_{i,j=1}^2$, $t_{ij} \in C S^0$, invertible for $\{|\xi| \geq M \}$, such that \begin{equation*}
	T^{-1}(t,x,\xi)A(t,x,\xi)T(t,x,\xi) = \begin{bmatrix}
		\lambda_1(t,x,\xi) & a_{12}(t,x,\xi) \\
		0 & \lambda_2(t,x,\xi)
	\end{bmatrix}
\end{equation*} for all $(t,x,\xi) \in [0,T] \times \R^n \times \{ |\xi| \geq M \}$.
\end{thm}

\begin{proof} For $2\times 2$ matrices the triangularisation procedure described in the previous subsection can stop at Step 1. By Remark \ref{rem:Transform}, we may assume that \eqref{cond:EVnonzero} holds for the eigenvector $h$ corresponding to $\lambda_1$ and for $j=1$. We set $h=h_1$ and $h^{(1)}= h_1$. The vector \begin{equation*}
		\omega_1 = \begin{bmatrix}
			\omega_{11}(t,x,\xi) \\ \omega_{12}(t,x,\xi)
		\end{bmatrix}, \quad \quad \omega_{1j}(t,x,\xi) = \frac{\la h^{(1)}(t,x,\xi) | e_j \ra}{\la h^{(1)}(t,x,\xi) | e_1 \ra},
	\end{equation*} belongs to $C S^0$ and is an eigenvector of $A$ associated to $\lambda_1$. We then set \begin{equation*}
	T_1(t,x,\xi) = \begin{bmatrix}
		\omega_1 & e_2
	\end{bmatrix} = \begin{bmatrix}
	\omega_{11}(t,x,\xi) & 0 \\
	\omega_{12}(t,x,\xi) & 1
\end{bmatrix}.
\end{equation*} With that, we obtain \begin{equation*}
A(t,x,\xi) T_1(t,x,\xi) = \begin{bmatrix}
	a_{11}\omega_{11} + a_{12}\omega_{12} & a_{12} \\
	a_{21}\omega_{11} + a_{22}\omega_{12} & a_{22}
\end{bmatrix}
\end{equation*} and finally, with \begin{equation*}
T_1^{-1}(t,x,\xi) = \begin{bmatrix}
	\omega_{11}(t,x,\xi) & 0 \\
	-\omega_{12}(t,x,\xi) & 1
\end{bmatrix},
\end{equation*} we obtain \begin{equation*} \begin{aligned}
	& T_1^{-1}(t,x,\xi)A(t,x,\xi)T_1(t,x,\xi) \\
	& \qquad = \begin{bmatrix}
		a_{11}\omega_{11}^2+a_{12}\omega_{12}\omega_{11} & a_{12}\omega_{11} \\
		-a_{11}\omega_{12}\omega_{11}-a_{12}\omega_{12}^2 + a_{21}\omega_{11} + a_{22}\omega_{12} & -\omega_{12}a_{12}+a_{22}
	\end{bmatrix}
\end{aligned}
\end{equation*} By construction, we have \begin{equation*} \begin{aligned}
	& a_{11}\omega_{11} + a_{12}\omega_{12} = \lambda_1 \omega_{11}, \\
	& a_{21}\omega_{11} + a_{22}\omega_{12} = \lambda_1 \omega_{12},
\end{aligned}
\end{equation*}
and $\omega_1=1$. This yields $a_{11}\omega_{11} + a_{12}\omega_{12} = \lambda_1 \omega_{11} = \lambda_1$ and \begin{equation*}
\begin{aligned}
	&-a_{11}\omega_{12}\omega_{11}-a_{12}\omega_{12}^2 + a_{21}\omega_{11} + a_{22}\omega_{12}\\
	& \qquad = -\omega_{12}(a_{11}\omega_{11}+a_{12}\omega_{12}) + a_{21}\omega_{11} + a_{22}\omega_{12} = -\lambda_1 \omega_{12} + \lambda_1 \omega_{12} = 0.
\end{aligned}.
\end{equation*} Using $a_{11}+a_{22} = \lambda_1 + \lambda_2$, we obtain \begin{equation*}
-\omega_{12}a_{12}+a_{22} = -\omega_{12}a_{12}+a_{22} + a_{11}\omega_{11} - a_{11}\omega_{11} = \lambda_2.
\end{equation*} Thus, we get that \begin{equation*}
T_1^{-1}(t,x,\xi) A(t,x,\xi) T_1(t,x,\xi)
= \begin{bmatrix}
	\lambda_1(t,x,\xi) & a_{12}(t,x,\xi) \\
	0 & \lambda_2(t,x,\xi)
\end{bmatrix}
\end{equation*} for $(t,x,\xi) \in [0,T] \times \R^n \times \{ |\xi| \geq M \}$. This concludes the proof.
\end{proof}

\subsubsection{Example}
\begin{itemize}
\item[(i)]
By direct computations we can easily see that if $h_1=[h_{11}\,\, h_{12}]^T=e_1$ then the matrix $A$ is automatically in the upper triangular form. Indeed,  
\[
a_{21}h_{11}+a_{22}h_{12}=\lambda_1 h_{12}
\]
implies $a_{21}=0$.  A typical example (already discussed in \cite{GramRuz2013}) is the Jordan block matrix 
\[
A=\begin{bmatrix}
	0 & 1\\
	0 & 0
\end{bmatrix},
\]
where $\lambda_1=0$ is an eigenvalue with eigenvector $h_1=e_1$.
\item[(ii)] Condition \eqref{cond:EVnonzero} is trivially fulfilled when ${\rm det A}\equiv 0$ and $A$ is of the form
\[
\begin{bmatrix}
	a & a\\
	-a & -a
\end{bmatrix},
\]
for $a=a(t,x,\xi)$. Indeed, also in this case one can take $0$ as an eigenvalue with eigenvector $h_1= [1 \,\, 1]^T$.
\end{itemize}

\subsection{The case $m=3$}

With the notation introduced in Section \ref{sec:Procedure}, we assume that the $3 \times 3$ matrix $A(t,x,\xi) = [a_{ij}(t,x,\xi)]_{i,j=1}^3$ admits three eigenvalues $\lambda_i(t,x,\xi) \in C S^1$, $i=1,2,3$, and two corresponding eigenvectors $h_i(t,x,\xi) \in \big( C S^0) \big)^3$, $i=1,2$. Then, we set $h^{(1)} := h_1$ and, as in Remark \ref{rem_gen_cond} we suppose that there is a $j_1 \in \{1,2,3\}$ with \begin{equation*}
	\la h^{(1)}(t,x,\xi) | e_{j_1} \ra \neq 0 \quad \forall (t,x,\xi) \in [0,T] \times \R^n \times \{ |\xi| \geq M \}.
\end{equation*} Thus, we can set \begin{equation*}
\omega_{1j}(t,x,\xi) = \frac{\la h^{(1)}(t,x,\xi) | e_j \ra}{\la h^{(1)}(t,x,\xi) | e_{j_1} \ra}.
\end{equation*} Now, we rearrange the matrix $A$ such that the first component of $\omega_1$ becomes identically equal to $1$. Then, with $j_2,j_3 \in \{ 1,2,3 \} \setminus \{ j_1 \}$, we can write \begin{equation*}
T_1^{-1} = \begin{bmatrix}
	\omega_1 & e_2 & e_3
\end{bmatrix}^{-1} = \begin{bmatrix}
\omega_{1j_1} & 0 & 0 \\
-\omega_{1j_2} & 1 & 0 \\
-\omega_{1j_3} & 0& 1
\end{bmatrix} = \begin{bmatrix}
1 & 0 & 0 \\
-\omega_{1j_2} & 1 & 0 \\
-\omega_{1j_3} & 0& 1
\end{bmatrix},
\end{equation*} which leads to \begin{equation*}
T_1^{-1} h_2 = \begin{bmatrix}
	\omega_{1j_1} & 0 & 0 \\
	-\omega_{1j_2} & 1 & 0 \\
	-\omega_{1j_3} & 0 & 1
\end{bmatrix} \begin{bmatrix}
h_{2j_1} \\ h_{2j_2} \\ h_{2j_3}
\end{bmatrix} =  \begin{bmatrix}
h_{2j_1} \\ -\omega_{1j_2}h_{2j_1} + h_{2j_2} \\ -\omega_{1j_3}h_{2j_1} + h_{2j_3}
\end{bmatrix}.
\end{equation*} We then get \begin{equation*}
h^{(2)} = \Pi_1 T_1^{-1} h_2 = \begin{bmatrix}
-\omega_{1j_2}h_{2j_1} + h_{2j_2} \\ -\omega_{1j_3}h_{2j_1} + h_{2j_3}
\end{bmatrix}
\end{equation*} and the condition \eqref{eq:CondThmSchur} that there exists $j \in \{ 1,2 \}$ such that \begin{equation*}
\la h^{(2)}(t,x,\xi) | e_{j} \ra \neq 0 \quad \forall (t,x,\xi) \in [0,T] \times \R^n \times \{ |\xi| \geq M \}
\end{equation*} translates to: either \begin{equation} \label{eq:aux12x2}
	-\omega_{1j_2}h_{2j_1} + h_{2j_2} \neq 0 \quad \Rightarrow \quad h_{2j_2} h_{1j_1} - h_{1j_2}h_{2j_1} \neq 0
\end{equation} or \begin{equation} \label{eq:aux22x2}
	-\omega_{1j_3}h_{2j_1} + h_{2j_3} \neq 0 \quad \Rightarrow \quad h_{2j_3}h_{1j_1} - h_{1j_3}h_{2j_1} \neq 0
\end{equation} holds. Thus, assuming that \eqref{eq:aux12x2} holds, the matrix $\tilde{T}_2$ is given by \begin{equation*}
	\begin{bmatrix}
	\omega_{21} & 0 \\
	\omega_{21} & 1
	\end{bmatrix} = \begin{bmatrix}
	1 & 0 \\
	\frac{-\omega_{1j_3}h_{2j_1} + h_{2j_3}}{-\omega_{1j_2}h_{2j_1} + h_{2j_2}} & 1
	\end{bmatrix}, \quad \omega_{2j} = \frac{\la h^{(2)}(t,x,\xi)|e_j\ra}{\la h^{(2)}(t,x,\xi)|e_{j_2}\ra}, \quad j=1,2,
\end{equation*} and the matrix $T_2$ by \begin{equation*}
	\begin{bmatrix}
	1 & 0 & 0 \\
	0 & \omega_{21} & 0 \\
	0 & \omega_{22} & 1 \\
	\end{bmatrix} = \begin{bmatrix}
	1 & 0 & 0 \\
	0 & 1 & 0 \\
	0& \frac{-\omega_{1j_3}h_{2j_1} + h_{2j_3}}{-\omega_{1j_2}h_{2j_1} + h_{2j_2}} & 1
	\end{bmatrix}.
\end{equation*} Thus, we obtain \begin{equation} \label{eq:auxTrafo}
T(t,x,\xi) = T_1T_2 = \begin{bmatrix}
1 & 0 & 0 \\
\omega_{1j_2} & 1 & 0 \\
\omega_{1j_3} & 0& 1
\end{bmatrix} \begin{bmatrix}
1 & 0 & 0 \\
0 & 1 & 0 \\
0& \frac{-\omega_{1j_3}h_{2j_1} + h_{2j_3}}{-\omega_{1j_2}h_{2j_1} + h_{2j_2}} & 1
\end{bmatrix}.
\end{equation} If we have \eqref{eq:aux22x2} instead of \eqref{eq:aux12x2}, then we would need a permutation matrix \begin{equation*}
	P_{j_{2} \leftrightarrow j_3} = \begin{bmatrix}
	1 & 0 & 0 \\
	0 & 0 & 1 \\
	0 & 1 & 0
	\end{bmatrix} \end{equation*} in \eqref{eq:auxTrafo}, i.e. \begin{equation*}
	T(t,x,\xi) = T_1(t,x,\xi)P_{j_{2} \leftrightarrow j_3} T_2(t,x,\xi)
\end{equation*} and \begin{equation*}
	T_2(t,x,\xi) = \begin{bmatrix}
	1 & 0 & 0 \\
	0 & 1 & 0 \\
	0& \frac{-\omega_{1j_2}h_{2j_1} + h_{2j_2}}{-\omega_{1j_3}h_{2j_1} + h_{2j_3}} & 1
	\end{bmatrix}.
\end{equation*} See also Remark \ref{rem_gen_cond}.



Thus, we can state \begin{thm}
	Suppose that $A(t,x,\xi) = [a_{ij}]_{i,j=1}^3$ admits three eigenvalues $\lambda_i \in C S^1$, $i=1,2,3$, and two corresponding eigenvectors $h_i(t,x,\xi) \in \big( C S^1 \big)^3$, $i=1,2$.  Suppose that there exists a $j_1 \in \{ 1, 2, 3\}$ such that \begin{equation} \label{cond_1_3}
		h_{1j_1}(t,x,\xi) \neq 0 \quad \forall (t,x,\xi) \in [0,T] \times \R^n \times \{ |\xi| \geq M \}.
	\end{equation} Further suppose that there exists $j_2 \in \{ 1,2,3\} \setminus \{ j_1 \}$ such that \begin{equation} \label{cond_2_3}
	h_{2j_2}h_{1j_1} - h_{1j_2} h_{2j_1} \neq 0 \quad \forall (t,x,\xi) \in [0,T] \times \R^n \times \{ |\xi| \geq M \}.
\end{equation} Then, there exists a matrix-valued symbol $T(t,x,\xi) = [t_{ij}]_{i,j=1}^3$, $t_{ij} \in C S^0$, invertible for all $(t,x,\xi) \in [0,T] \times \R^n \times \{ |\xi| \geq M \}$, such that \begin{equation*}
T^{-1}(t,x,\xi)A(t,x,\xi)T(t,x,\xi) = \Lambda(t,x,\xi) + N(t,x,\xi)
\end{equation*} holds for all $(t,x,\xi) \in [0,T] \times \R^n \times \{ |\xi| \geq M \},$ where $\Lambda(t,x,\xi) = \diag(\lambda_1, \lambda_2, \lambda_3)$ and \begin{equation*}
N(t,x,\xi) = \begin{bmatrix}
	0 & N_{13}(t,x,\xi) & N_{13}(t,x,\xi) \\
	0 & 0 & N_{23}(t,x,\xi) \\
	0 & 0 & 0
\end{bmatrix}.
\end{equation*}
\end{thm}

We end this subsection by discussing some examples of  $3\times 3$ matrices fulfilling the assumptions above on their eigenvalues. 

\subsubsection{Examples}
\begin{itemize}
\item[(i)] If the matrix $A$ has eigenvectors \begin{equation*}
	h_1= \begin{bmatrix}
	1\\ 0\\ 1
	\end{bmatrix} \quad \text{and} \quad h_2= \begin{bmatrix}
	1\\ 1\\	0
	\end{bmatrix}
\end{equation*} then conditions \eqref{cond_1_3} and \eqref{cond_2_3} are easily fulfilled with $j_1=1$ and $j_2=2$. Indeed, $h_{11}=1$ and
\[
h_{22}h_{11}-h_{12}h_{21}=h_{22}h_{11}=1.
\]
More in general to satisfy \eqref{cond_1_3} and \eqref{cond_2_3}  it would be enough to have two eigenvectors \begin{equation*}
	h_1= \begin{bmatrix}
	h_{11}\\ h_{12}\\ h_{13}
	\end{bmatrix} \quad \text{and} \quad h_2= \begin{bmatrix}
	h_{21}\\ h_{22}\\ h_{23}
	\end{bmatrix}
\end{equation*} with $h_{11}\neq 0$, $h_{22}\neq 0$ and $h_{12}=0$.

\item[(ii)] A matrix with eigenvectors \begin{equation*}
	h_1= \begin{bmatrix}
	1\\ 0\\ 1
	\end{bmatrix} \quad \text{and} \quad h_2= \begin{bmatrix}
	1\\ 1\\ 0
	\end{bmatrix}
\end{equation*} has a special form. Indeed, for $\lambda_1$ and $\lambda_2$ eigenvalues corresponding to $h_1$ and $h_2$, respectively, by using the eigenvector equations we obtain
\[
\begin{split}
a_{13}&=\lambda_1-a_{11},\\
a_{23}&=-a_{21},\\
a_{33}&=\lambda_1-a_{31},
\end{split}
\]
and
\[
\begin{split}
a_{12}&=\lambda_2-a_{11},\\
a_{22}&=\lambda_2-a_{21},\\
a_{32}&=-a_{31}.
\end{split}
\]
Hence
\[
A=\begin{bmatrix}
	a_{11} & \lambda_2-a_{11} & \lambda_1-a_{11}\\ 
	a_{21} & \lambda_2-a_{21} & -a_{21}\\
a_{31} & -a_{31} & \lambda_1-a_{31} 
\end{bmatrix}. 
\]

\end{itemize}


\begin{thebibliography}{99}

	\bibitem{Bernstein} D. S. Bernstein.
	\newblock \emph{Matrix Mathematics -- Theory, Facts, and Formulas}.
	Princeton University Press, 2nd edition, (2009).	
	
	\bibitem{Bronshtein} M. D. Bronshtein.
	\newblock {Smoothness of roots of polynomials depending on parameters}.
	\newblock \emph{Sibirsk. Mat. Zh.}, 20(3), 493--501, (1979), English transl. in \emph{Siberian Math. J.}, 20, 347--352, (1980).
	
	\bibitem{ColKi:02}
{ F.~Colombini and T.~Kinoshita.}
\newblock On the Gevrey well posedness of the Cauchy problem for weakly hyperbolic equations of higher order.
\newblock {\em J. Diff. Eq.}, 186, 394--419, (2002).

\bibitem{ColKi:02-2}
{ F.~Colombini and T.~Kinoshita.}
\newblock
On the Gevrey wellposedness of the Cauchy problem for weakly hyperbolic equations of 4th order.
\newblock {\em Hokkaido Math. J.}, 31, 39--60, (2002).

\bibitem{COr}
{ F.~Colombini and N.~Orr\`u}.
\newblock
Well-posedness in $C^\infty$ for some weakly hyperbolic equations.
\newblock {\em J. Math. Kyoto Univ.}, 39, 399--420, (1999).

\bibitem{CS}
{F. Colombini and S. Spagnolo}.
An example of a weakly hyperbolic Cauchy problem not well posed in $C^\infty$.
{\em Acta Math.}, 148, 243--253, (1982).
 

\bibitem{CDS}
F. Colombini, E. De Giorgi and S. Spagnolo.
\newblock {Sur les \'equations hyperboliques avec des coefficients qui ne d\'ependent que du temps}.
\newblock \emph{ Ann. Scuola Norm. Sup. Pisa Cl. Sci.}, 6,  511--559, (1979).

\bibitem{ColMet:1}
F. Colombini, D. Del Santo, F. Fanelli and G. M\'etivier.
\newblock Time-dependent loss of derivatives for hyperbolic operators with non regular coefficients. 
{\em Comm. Partial Diff. Eq.}, 38(10), 1791--1817, (2013). 

\bibitem{ColMet:2}
F. Colombini, D. Del Santo, F. Fanelli and G. M\'etivier.
\newblock A well-posedness result for hyperbolic operators with Zygmund coefficients.
{\em J. Math. Pures Appl.}, 9(100) 455--475, 2013. 

\bibitem{CJS}
F. Colombini, E. Jannelli and S. Spagnolo.
\newblock {Well-posedness in the Gevrey classes of the Cauchy problem for a nonstrictly hyperbolic equation with coefficients depending on time}. 
{\em Ann. Scuola Norm. Sup. Pisa Cl. Sci. }, 10,  291--312, (1983).

\bibitem{CJS2}
F. Colombini, E. Jannelli and S. Spagnolo.
\newblock {Nonuniqueness in hyperbolic Cauchy problems}.
{\em Ann. Math.} 126, 495--524, (1987).

\bibitem{ColLer}
F. Colombini and N. Lerner.
\newblock {Hyperbolic operators with non-Lipschitz coefficients}.
{\em  Duke Math. J.}, 77(3),  657--698, (1995).

\bibitem{ColNi}
F. Colombini and T. Nishitani.
\newblock  Second order weakly hyperbolic operators with coefficients sum of powers of functions. 
{\em Osaka J. Math. }44(1), 121--137, (2007). 


\bibitem{dAKi:05}
{P. D'Ancona and T. Kinoshita}.
On the wellposedness of the Cauchy problem for weakly hyperbolic equations of higher
order.
{\em Math. Nachr.}, 278, 1147--1162, (2005).

\bibitem{dAKiS:04}
{P. D'Ancona, T. Kinoshita and S. Spagnolo}.
\newblock{Weakly hyperbolic systems with H\"older continuous coefficients.}
\newblock{\em J. Diff. Eq.}, 203(1), 64--81, (2004).


\bibitem{dAKiS:08}
{P. D'Ancona, T. Kinoshita and S. Spagnolo}.
\newblock On the 2 by 2 weakly hyperbolic systems.
\newblock{\em Osaka J. Math.}, 45(4), 921--939, (2008).
 

\bibitem{DS}
P. D'Ancona and S. Spagnolo.
\newblock {Quasi-symmetrisation of hyperbolic systems and propagation of the analytic regularity}. 
\newblock \emph{Boll. UMI}, 8(1B), 169--185, (1998).
	
\bibitem{Dieci} L. Dieci and T. Eirola. 
\newblock {On smooth decompositions of matrices}.
\newblock \emph{ SIAM J. Matrix Anal. Appl.}, 20(3), 800--819, (1999).  
	
	\bibitem{Dieci2014} L. Dieci, A. Papini, A. Pugliese, and A. Spadoni.
	\newblock {Continuous Decompositions and Coalescing Eigenvalues for Matrices Depending on Parameters}. In \emph{Current Challenges in Stability Issues for Numerical Differential Equations}, 173-264, Lecture Notes in Mathematics 2082, Springer, (2014).  
	
	\bibitem{Duis}
J. J. Duistermaat.
\emph{Fourier intergal operators.}
Progress in Mathematics, 130. Birkh\"auser Boston, Inc., Boston, MA, (1996).
	
	\bibitem{G:15}
C.~Garetto.
\newblock On hyperbolic equations and systems with non-regular time dependent coefficients
\newblock {\em J. Diff. Eq.}, 259(11), 5846-5874, (2015).

\bibitem{GJ}
C.~Garetto and C. J\"ah.
\newblock Well-posedness of hyperbolic systems with multiplicities and smooth coefficients.
{\em Math. Ann.}, 369(1-2), 441-485,  (2017).
 
\bibitem{GarRuz:1}
C.~Garetto and M.~Ruzhansky.
\newblock Well-posedness of weakly hyperbolic equations with time dependent coefficients.
\newblock {\em J. Diff. Eq.}, {253(5)},1317--1340, (2012).

\bibitem{GR}
C.~Garetto and M.~Ruzhansky.
\newblock Weakly hyperbolic equations with non-analytic 
coefficients and lower order terms.
\newblock {\em Math. Ann.}, 357(2), 401--440, (2013).

\bibitem{GarRuz:3}
C.~Garetto and M.~Ruzhansky.
A note on weakly hyperbolic equations with analytic principal part.
\newblock {\em J. Math. Anal. Appl.}, 412(1):1--14, (2014).

\bibitem{GarRuz:ARMA}
C.~Garetto and M.~Ruzhansky.
Hyperbolic second order equations with non-regular time dependent coefficients. 
{\em Arch. Ration. Mech. Anal.}, 217(1):113--154, (2015). 

\bibitem{GarRuz:17-1}
C. Garetto and M. Ruzhansky.
\newblock On hyperbolic systems with time dependent H\"older characteristics.
\newblock {\em Ann. Mat. Pura  Appl}., 196(1), 155-164, (2017).

\bibitem{GarRuz:17-2}
C. Garetto and M. Ruzhansky.
\newblock On $C^\infty$ well-posedness of hyperbolic systems with multiplicities.
\newblock {\em Ann. Mat. Pura  Appl}., 196(5), 1819-1834,  (2017).

	
	\bibitem{Ginggold1979} H. Gingold.
	\newblock {On continuous triangularization of matrix functions}.
	\newblock {\em SIAM J. Math. Anal.}, 10(4), 709--720, (1979).  
	
	\bibitem{Ginggold1992} H. Gingold.
	\newblock {Globally analytic triangularization of a matrix function}.
	\newblock {\em Linear Algebra Appl.}, 169, 75--101, (1992). 
	
	\bibitem{GramOrr2011} T. Gramchev and N. Orr\'u.
	\newblock Cauchy problem for a class of nondiagonalizable hyperbolic systems.
	\newblock {\em Discrete Contin. Dyn. Syst.}, 533--542, (2011).
	
	\bibitem{GramRuz2013} T. Gramchev and M. Ruzhansky.
	\newblock {Cauchy problem for $2 \times 2$ hyperbolic systems of pseudo-differential equations with nondiagonalisable principal part}.
	\newblock {\em Studies in Phase Space Analysis with Applications to PDEs}, Progr. Nonlinear Differential Equations Appl., 84, 129--144, (2013)
	
	 \bibitem{Hor:85} L. H\"{o}rmander. The analysis of linear
     partial differential operators.  Vols. I-IV, Springer-Verlag, (1985).
     
        \bibitem{Hord} {L. H\"ormander}.
    Hyperbolic systems with
     double characteristics, 
     {\em Comm. Pure Appl. Math.}, 46, 261--301, (1993).
     
     \bibitem{IvPet}
{V. Ya. Ivrii and V. M. Petkov}.
Necessary conditions for the correctness of the Cauchy problem
for non-strictly hyperbolic equations.
(Russian) Russian Math. Surveys, 29, 3--70, (1974).
     
	
	\bibitem{KY:06}
K.~Kajitani and Y.~Yuzawa.
\newblock The {C}auchy problem for hyperbolic systems with {H}{\"o}lder
  continuous coefficients with respect to the time variable.
\newblock {\em Ann. Sc. Norm. Super. Pisa Cl. Sci. (5)}, 5(4):465--482, (2006).

\bibitem{KR1}
{I. Kamotski and  M. Ruzhansky}.
Estimates and spectral asymptotics for systems with multiplicities. 
{\em Funct. Anal. Appl.}, 39, 308--310, (2005).

\bibitem{KR2}
{I. Kamotski and  M. Ruzhansky}.
Regularity properties, representation of solutions and 
spectral asymptotics of systems with multiplicities, 
{\em Comm. Partial Diff. Eq.}, 32, 1--35, (2007).
	
	\bibitem{Kato} T. Kato.
	\newblock \emph{Perturbation Theory for Linear Operators}. 2nd ed., Springer, Berlin, (1976).
	
	\bibitem{KS:06}
T. Kinoshita and S. Spagnolo,
\newblock{Hyperbolic equations with non-analytic coefficients}.
\newblock {\em Math. Ann.}, 336, 551--569, (2006).
	
	\bibitem{Lax:57} P. Lax. Asymptotic solutions of oscillatory initial value problems. 
{\em Duke Math. J.}, 24, 627--646, (1957).

\bibitem{MU:2}
{R. B. Melrose and G. A. Uhlmann}.
{Microlocal structure of involutive
conical refraction}, {\em Duke Math. J.}, 46, 571--582, (1982).


\bibitem{Ni:06}
{T. Nishitani}
{On the Cauchy problem for $D^2_t-D_x a(t,x)^n D_x$.}
{\em  Ann. Univ. Ferrara Sez. VII Sci. Mat.},  52(2), 395--430, (2006). 

\bibitem{OT:84}
{Y. Ohya and S. Tarama.}
\newblock The Cauchy Problem with multiple characteristics in the Gevery class - H\"older coefficients in $t$.
\newblock{\em Hyperbolic Equations and related topics. Kataka/Kioto}, 273--306, (1984).


\bibitem{PP}
{C. Parenti and  A. Parmeggiani}.
On the Cauchy problem for hyperbolic operators with double characteristics. 
{\em Comm. Partial Diff. Eq.}, 34, 837--888, (2009).
	
	\bibitem{Parusinski} A. Parusinski and A. Rainer.
	\newblock {Regularity of roots of polynomials}.
	\newblock {\em Ann. Sc. Norm. Super. Pisa Cl. Sci}. (5), 16(2), 481--517, (2016). 
	
	
	\bibitem{RellichOld} F. Rellich.
	\newblock {St\"rungstheorie der Spektralzerlegung I-V}. 
	{\em Math. Ann.}, (113)-(118), (1936)-(1942).
	
	\bibitem{Rellich} F. Rellich.
	\newblock \emph{Perturbation Theory of Eigenvalue Problems}. Gordon and Breach, New York, (1969).
	
	 \bibitem{Roz}
G. Rozenblum. Spectral asymptotic behavior of elliptic systems. (Russian) 
{\em Zap. LOMI}, 96, 255--271, (1980).

\bibitem{Ruz}
M. Ruzhansky. 
Singularities of affine fibrations in the theory of regularity of Fourier integral operators. 
{\em Russian Math. Surveys}, 55(1), 93--161, (2000).

\bibitem{Ruz-CWI}
M. Ruzhansky. 
{\em Regularity theory of Fourier integral operators with complex phases and singularities of affine fibrations}. CWI Tract, 131. Stichting Mathematisch Centrum, Centrum voor Wiskunde en Informatica, Amsterdam, (2001). 
	
	\bibitem{Yu:05}
Y.~Yuzawa.
\newblock The {C}auchy problem for hyperbolic systems with {H}{\"o}lder
  continuous coefficients with respect to time.
\newblock {\em J. Diff. Eq.}, 219(2), 363--374, (2005).
	
	\bibitem{Wasow1962} W. Wasow. 
	\newblock {On holomorphically similar matrices}.
	\newblock {\em J.  Math. Anal. Appl.}, 4(2), 202--206, (1962).  
	
\end{thebibliography}
\end{document}